\documentclass[12pt, reqno]{amsart}
\usepackage{amsmath}

\usepackage{algorithm}
\usepackage[noend]{algpseudocode}

\usepackage{graphicx}
\usepackage{subcaption}
\usepackage{hyperref}

\usepackage{xcolor}
\allowdisplaybreaks

\newtheorem{theorem}{Theorem}[section]

\theoremstyle{definition}

\theoremstyle{remark}
\newtheorem{remark}[theorem]{Remark}
\numberwithin{equation}{section}

\DeclareMathOperator*{\argmin}{arg\!\min}

\newcommand{\secref}[1]{Section~\ref{#1}}

\newcommand\numberthis{\addtocounter{equation}{1}\tag{\theequation}}

\numberwithin{equation}{section}
\renewcommand{\theequation}{\arabic{section}.\arabic{equation}}

\begin{document}
\title[Optimal Control on COVID-19] % running head version
{Data-driven Optimal Control of\\ a SEIR model for COVID-19}
\author{Hailiang Liu and Xuping Tian}
\address{Iowa State University, Mathematics Department, Ames, IA 50011} \email{hliu@iastate.edu, xupingt@iastate.edu}
%\date{\today}
\begin{abstract}
%Since the first suspected case of COVID-19 in December 2019, a total of $11,357,322$ confirmed cases in the US and $55,624,562$ confirmed cases worldwide have been reported, up to November 17, 2020, evoking fear locally and internationally. %In particular, the coronavirus is surging nationwide at this time.
We present a data-driven optimal control approach which integrates the reported partial data with the epidemic dynamics
for COVID-19. We use a basic Susceptible-Exposed-Infectious-Recovered (SEIR) model, the model parameters are time-varying and learned from the data. This approach serves to forecast the evolution of the outbreak over a relatively short time period and provide scheduled controls of the epidemic. We provide efficient numerical algorithms based on a generalized Pontryagin Maximum Principle associated with the optimal control theory. Numerical experiments demonstrate the effective performance of the proposed model and its numerical approximations.
%We further discuss refined models such as the multi-group model and the one with spatial movements. We integrate the health care support and social distance practice with the basic model, and analyze other models to understand the transmission dynamics.
\end{abstract}

\maketitle
%\tableofcontents

\section{Introduction}
The outbreak of COVID-19 epidemic has resulted in over millions of confirmed and death cases, evoking fear locally and internationally. It has a huge impact on global economy as well as everyone's daily life. Numerous mathematical models are being produced to forecast the spread of COVID-19 in the US and worldwide \cite{AR20, GB20, BF20, LZ20}. These predictions have far-reaching consequences regarding how quickly and how strongly governments move to curb the epidemic. %Mathematical models can be profoundly helpful tools to make public health decisions and ensure optimal use of resources to reduce the morbidity and mortality associated with the COVID-19 pandemic, but only if they are rigorously evaluated and validated and their projections are robust and reliable.
We aim to exploit the abundance of available data and integrate existing data with disease dynamics based on epidemiological knowledge.

\iffalse
The outbreak of COVID-19 epidemic has resulted in over millions of confirmed and death cases, evoking fear locally and internationally. It has a huge impact on global economy as well as everyone's daily life. Mathematical models can be profoundly helpful tools to make public health decisions and ensure optimal use of resources to reduce the morbidity and mortality associated with the COVID-19 pandemic, but only if they are rigorously evaluated and validated and their projections are robust and reliable. Numerous mathematical models are being produced to forecast the spread of COVID-19 in the US and worldwide \cite{AR20, GB20, BF20, LZ20}. We aim to exploit the abundance of available data and integrate existing data with disease dynamics based on epidemiological knowledge.
\fi

One of the well-known dynamic models in epidemiology is the SIR model proposed by Kermack and McKendrick \cite{KM27} in 1927. Here, $S, I, R$ represent the number of susceptible, infected and recovered people respectively. They use an ODE system to describe the transmission dynamics of infectious diseases among the population. In the current COVID-19 pandemic, actions such as travel restrictions, physical distancing and self-quarantine are taken to slow down the spread of the epidemic. Typically, there is a significant incubation period during which individuals have been infected but are not yet infectious themselves. During this period the individual is in compartment $E$ (for exposed), the resulting models are of SEIR or SEIRS type, respectively, depending on whether the acquired immunity is permanent or otherwise. Also such models can show how different public health interventions may affect the outcome of the epidemic, and can also predict future growth patterns.

Optimal control provides a perspective to study and understand the underlying transmission dynamics. The classical  theory  of  optimal  control  was  originally  developed to  deal  with systems  of  controlled  ordinary  differential  equations \cite{He66, BP07}. Computational methods have been designed to solve related control problems \cite{KC62, CL82, Sc06}. There is a wide application to various fields, including those for epidemic models, with major control measures on medicine (vaccination), see e.g., \cite{JK20, Be00}, and for Cancer immunotherapy control \cite{CP07}.

In this paper, we integrate the optimal control with a specific SEIR model, though the developed methods can be readily adapted to other epidemic ODE models. More precisely, we introduce a dynamic control model for monitoring the virus propagation. Here the goal is to advance our understanding of virus propagation by learning the model parameters so that the error between the reported cases and the solution to the SEIR model is minimized.
%The developed teachinuqe can be
%minimize the number of infection and death agents of the population or the error to the reported cases for the training of the model.
%\red{minimize the error between the reported cases and the solution to the SEIR model.}
In short, we formulate the  following optimization problem
\begin{align*}
  \min_\theta \quad &J = \sum_{i=1}^{n-1}L(U(t_i)) + g(U(T)),\\
  \text{s.t.}\quad &\dot{U}=F(U,\theta) \quad t\in(0,T], \quad U(0)=U_0, \quad \theta\in\Theta.
\end{align*}
Here $\dot U=\frac{d}{dt}U(t)$, $\theta$ is a time-varying vector of model parameters, $\Theta$ is the admissible set for model parameters $\theta$, and $U=[S, E, I, R, D]^\top$ corresponds to the susceptible, exposed, infected,  recovered and deceased population. The loss function $J$ is composed of two parts: $L$ measures the error between the candidate solution to the SEIR system and the reported data at intermediate observation times, and $g$ measures the error between the candidate solution and the scheduled control data at the end time. This is in contrast to the standard approach of fitting the model to the epidemic data through the non-linear least squares approach \cite{AR20, GB20, BF20}.

\subsection{Main contributions}
In the present work we provide a fairly complete modeling discussion on parameter learning, prediction and control of epidemics spread based on the SEIR model. We begin with our discussion of basic properties of the SEIR model. Then we show how the parameters can be updated recursively by gradient-based methods for a short time period while parameters are near constants, where the loss gradient can be obtained by using both state and co-state dynamics. For an extended time period with reported data available at intermediate observation times, the parameters are typically time-varying. In this general setting, we derive necessary conditions to achieve optimal control in terms of the chosen objectives. The conditions are essentially the classical Pontryagin maximum principal (PMP)\cite{PB62}. The main differences are in the way we apply the principle in each time interval, and connect them consistently by re-setting the co-state $V$ at the end of each time interval. We thus named it the generalized PMP. We further present an algorithm to find the numerical solution to the generalized PMP in the spirit of the method of successive approximations (MSA)\cite{CL82}. Our algorithm is mainly fulfilled by three parts: (1) discretization of the forward problem in such a way that solutions to $U$ remain positive for an arbitrarily step size $h$; (2) discretization of the co-state equation for $V$ is made unconditionally stable; (3) Minimiziation of the Hamiltonian is given explicitly based on the structure of the SEIR model.

This data-driven optimal control approach can be applied to other epidemic models. In particular, the prediction and control can be combined into one framework. To this end, the cost includes terms measuring the error between confirmed cases (infection and death) and those predicted from the model during the evolution, and terms measuring the error between the scheduled numbers as desired and those predicted from the model during the control period.

%Finally, the data-driven optimal control is approximated in three stages: (1) discretization of the forward problem in such a way that solutions to $U$ remain positive for an arbitrarily step size $h$; (2) discretization of the co-state equation for $V$ is made unconditionally stable; (3) Minimiziation of the Hamiltonian is given explicitly based on the structure of the SEIR model.

\subsection{Further related work}
In the mathematical study of SIR models, there is an interplay between the dynamics of the disease and that of the total population, see \cite{AM79, MH92, GD95, Th92}. We refer the reader to \cite{LH87, Gr97} for references on SEIRS models with constant total population and to \cite{LM95} for the proof of the global stability of a unique endemic equilibrium of a SEIR model. Global stability of the endemic equilibrium for the SEIR model with non-constant population is more subtle, see \cite{LG99}. Apart from the compartmental models (and their stochastic counterparts \cite{A08}), a wide variety of methods exist for modeling infectious disease. These include diffusion models \cite{Ca10}, mean-field-type models \cite{Kl99}, cellular automata \cite{WRS06}, agent-based models \cite{HNK18}, network-based models \cite{WS98, KE05, PCH10}, and game-theoretical modelling \cite{BE04, RG11, Ch20}. Some focus on the aggregate behaviour of the compartments of the population, whereas others focus on individual behaviour.

\iffalse  [ Add these refs ..]
 I have added a few more as last addition on this part.

For all these models, a common and challenging problem is how to learn the time-varying parameters from the data. [4] builds a neural network that output the parameters directly from the data, [3] express the parameters as learnable functions (in the form of neural network) with respect to time-varying covariates, then the parameters are fed into the SEIR model and are learned by training the jointed neural network - SEIR model. [5] also incorporates neural networks with SEIR model and applied training methods as introduced in [6]. Compared to such works, we consider the SEIR model itself as a learnable function and provide algorithms to learn the parameters from the view of optimal control. And unlike [7] that focus on comparing the effects of different possible control strategies, our framework provides strategies that meet desired outcomes with scheduled control.
\fi

For COVID-19, a data-driven model has been proposed and simulated in \cite{MD20}, where both the SIR model and the feed-forward network are trained jointly. Compared to this work, our data-driven model is to consider the optimal control for SEIR with rigorous derivation of optimality conditions and stable numerical approximations. On the other hand, our data-driven optimal control algorithm may be interpreted as training a deep neural network in which the SEIR model serves as the neural network with parameters to be learned. This is in contrast to the study of residual neural networks using neural ODEs (see e.g., \cite{E17, CR18}) or section dynamics \cite{LM20}. For other works on data-driven learning model parameters using neural networks, see \cite{AL20,JS20}. SIR models with spatial effects have been studied \cite{BR20, LL20}. A SIRT model was proposed in \cite{BR20} to study the effects of the presence of a road on the spatial propagation of the COVID-19. Introduced in \cite{LL20} is a mean-field game algorithm for a spatial SIR model in controlling propagation of epidemics.

Finally, we would like to mention that a variety of works have sought to endow neural networks with physics-motivated equations, which can improve machine learning. To incorporate the equations, parametric approaches have been studied \cite{GDY19, CG_20, LR19, RP19}, where certain terms in the equations are learned, while the equations still govern the dynamics from input to output.
%SIR models can also be modified to reflect the effect on the disease progression by the policy changes such as travel bans or public restrictions.

\subsection{Organization}
In \secref{ch:SEIR}, we motivate and present the SEIR system for modelling the time evolution of epidemics, and integrate it with data into an optimal control system. Both algorithms and related numerical approximations are detailed in \secref{ch:alg}. \secref{ch:exps} provides experimental results to show the performance of our data-driven optimal control algorithms. We end with concluding remarks in \secref{ch:con}.

%\subsection{Notation}
%Throughout this paper, we use $\partial_u$ to denote the gradient  with respect to $u$ when $u$ is an vector.

\section{The SEIR model and optimal control}\label{ch:SEIR}
\subsection{Model formulation}
A population of size $N(t)$ is partitioned into subclasses of individuals who are susceptible, exposed (infected but not yet infectious), infectious, and recovered, with sizes denoted by $S(t), E(t), I(t), R(t)$, respectively. The sum $E+I$ is the total infected population.

The dynamical transfer of the population is governed by an ODE system
%\begin{subequations}
%\renewcommand{\theequation}{\theparentequation.\arabic{equation}}
\begin{align*}
& \dot{S} = A -\beta SI/N  -dS, \\
& \dot{E} =  \beta SI/N - \epsilon E -dE, \\
& \dot{I} = \epsilon E  -\mu I -\gamma I-dI, \\
& \dot{R} = \gamma I -dR, \\
& \dot{D} =\mu I,\numberthis \label{eq:death}
\end{align*}
%\end{subequations}
subject to initial conditions $S_0, E_0, I_0, R_0, D_0$. Here $D(t)$  denotes deaths at time $t$,  $A$ denotes the recruitment rate of the population due to birth and immigration. It is assumed that all newborns are susceptible and vertical transmission can be neglected. $d$ is the natural death rate, $\mu$ is the rate for virus-related death. $\gamma$ is the rate for recovery with $1/\gamma$ being the mean infectious period,  and $\epsilon$ is the rate at which the exposed individuals become infectious with $1/\epsilon$ being the mean incubation period. The recovered individuals are assumed to acquire permanent immunity (yet to be further confirmed for COVID-19); there is no transfer from the $R$ class back to the $S$ class. $\beta$ is the effective contact rate. In the limit when $\epsilon \to \infty$, the SEIR model becomes a SIR model. Note that the involved parameters do not correspond to the actual per day recovery and mortality rates as the new cases of recovered and deaths come from infected cases several days back in time. However, one can attempt to provide some coarse estimations of the ``effective/apparent" values of these epidemiological parameters based on the reported confirmed cases.

\subsection{Solution properties of the SEIR model} Classical disease transmission models typically have at most one endemic equilibrium. If there is no endemic equilibrium, diseases will disappear. Otherwise, the disease will be persistent irrespective of initial positions. Large outbreaks tend to the persistence of an endemic state and small outbreaks tend to the extinction of the diseases.

To understand solution properties of the SEIR system, we simply take $A=bN$ with $b$ the natural birth rate, and focus on the following sub-system
\begin{align*}
& %\frac{d}{dt}
\dot S = b N -\beta SI/N  -dS, \\
& %\frac{d}{dt}
\dot E =  \beta SI/N - \epsilon E -dE, \\
& %\frac{d}{dt}
\dot I= \epsilon E  -\mu I -\gamma I-dI, \\
& %\frac{d}{dt}
\dot R= \gamma I -dR
\end{align*}
with the total population size $N=S+E+I+R$. By adding the equations above we obtain
$$
\dot N=(b-d)N-\mu I.
$$
Let $s=S/N, e=E/N, i=I/N, r=R/N$ denote the fractions of the classes $S, E, I, R$ in the population, respectively, then   one can verify that \begin{subequations}\label{new}
\begin{align}
& %\frac{d}{dt}
\dot s = b-bs  -\beta is+\mu is, \\
& %\frac{d}{dt}
\dot e =  \beta is - (\epsilon+b)e +\mu i e, \\
& %\frac{d}{dt}
\dot i= \epsilon e  -(\mu+\gamma +b)i+\mu i^2,
\end{align}
\end{subequations}
where $r$ can be obtained from  $r=1- s-e-i$ or
$$
\dot r= \gamma i  -br +\mu i r.
$$
From biological considerations,  we study system (\ref{new}) in a feasible region
$$
\Sigma=\{(s, e, i)\in \mathbb{R}^3_+, 0\leq s+e+i \leq 1\}.
%\Sigma=\{(S, E, I, R)\in \mathbb{R}^4_+, 0\leq S \leq \bar S, \; 0< N\leq \bar N\},
$$
%with $\bar S= A/d$ and $\bar N=A/d$. Since $\dot S \leq A-dS$ and  $\dot N \leq A-dN$,
It can be verified that $\Sigma$ is positively invariant with respect to the underlying dynamic system. We denote $\partial \Sigma$ and $\Sigma^0$ the boundary and the interior
of $\Sigma$ in $\mathbb{R}^3$, respectively. A special solution of form $P^0=(1, 0, 0)$ on the boundary of $\Sigma$ is the disease-free equilibrium of system (\ref{new}) and it exists for all non-negative values of its parameters. Any equilibrium in the interior $\Sigma^0$ of $\Sigma$ corresponds to the disease being endemic and is named an endemic equilibrium.
%is called an endemic equilibrium if $P^*$ satisfy equilibrium equations:
%\begin{align*}
%&  A -\beta S^*I^*/N^*  -dS^*=0, \\
%&   \beta S^*I^*/N^* - \epsilon E^* -dE^*=0, \\
%&  \epsilon E^*  -\mu I^* -\gamma I^*-d I^*=0, \\
%& \gamma I^* -dR^*=0.
%\end{align*}
%With the above setting, a key quantity is
%$$
%\sigma= \frac{\beta \epsilon}{(\epsilon+b)(\gamma+\mu+b)},
%$$
%called the modified contact number %\cite{LGWK99, LG99}.
To identify the critical threshold for the parameters, we take $\epsilon \times$(\ref{new}b)+ $(\epsilon+b)\times$(\ref{new}c) to obtain
\begin{align*}
\frac{d}{dt}[\epsilon e +(\epsilon+b)i]
&= i \left[\epsilon \beta s +\epsilon \mu e +\mu (\epsilon +b)i-(\epsilon+b)(\gamma+\mu+b)
\right] \\
& \leq i\left[  \max\{ \epsilon \beta,
\epsilon \mu, \mu (\epsilon +b)\} -(\epsilon+b)(\gamma+\mu+b) \right] \leq 0,
\end{align*}
provided a key quantity $\sigma \leq 1$, where
$$
\sigma : = \frac{\beta \epsilon}{(\epsilon+b)(\gamma+\mu+b)}
$$
is called the modified contact number \cite{LG99}. %LGWK99,
The global stability of $P^0$ when $\sigma \leq 1$
follows from LaSalle's invariance principle. The following theorem (see \cite{LG99}), a standard type in mathematical epidemiology, shows that the basic number $\sigma$ determines the long-term outcome of the epidemic outbreak.

%[LGWK99] M.Y. Li, J.R. Graef, L. Wang, J. Karsai, Global dynamics of a SEIR model with varying total population size, Math. Biosci. 160 (1999) 191–213.

\begin{theorem} %Let $R_0=\frac{\beta}{d+\mu+\gamma}$ be  the basic reproduction number. \\
1. If $\sigma \leq 1$, then $P^0$ is the unique equilibrium and globally stable in $\Sigma$.\\
2. If $\sigma>1$, then $P^0$ is unstable and the system is uniformly persistent in $\Sigma^0$.
\end{theorem}
By this theorem, the disease-free equilibrium $P^0$ is globally stable in $\Sigma$ if and only if $\sigma \leq 1$. Its epidemiological implication is that the infected fraction of the population vanishes in time so the disease dies out. In addition, the number $\sigma$ serves as a sharp threshold parameter;  if $\sigma>1$, the disease remains endemic.

In the epidemic literature, another threshold quantity is the basic reproduction number
$$
R_0= \frac{\beta}{\gamma+\mu},
$$
which is the product of the contact rate $\beta$ and the average infectious period
$\frac{1}{\gamma+\mu}$. %only defined at the time of invasion.
It is a parameter well known for quantifying the epidemic spread \cite{DH99, DW02}. On the other hand, the contact number $\sigma$ is defined at all times. In general, we have
$$
R_0 \geq \sigma.
$$
Note that for most models, $\sigma=R_0$, both quantities can be used interchangeably. This is the case in our experiments when taking $b=0$.
%$$
%R_0= \frac{\beta}{\gamma+\mu}.
%$$

%[DH99] O. Diekmann, J.A.P. Heesterbeek, Mathematical Epidemiology of Infectious Diseases: Model Building, Analysis and Interpretation, Wiley, New York, 1999.

%[DW02] P. van den Driessche, J. Watmough, Reproduction numbers and sub-threshold endemic equilibria for compartmental models of diseases transmission, Math. Biosci. 180 (2002) 29–48.

%Proof is deferred to appendix.
%In the case $R_0>1$, we have the global stability about $P^*$.
%\begin{theorem} If $R_0>1$, then there exists a unique endemic equilibrium $P^*$ for the system, and $P^*$
%is globally asymptotically stable in $\Sigma^0$.
%\end{theorem}

\begin{remark} Extensive evidence shows that the disease spread rate is sensitive to at least three factors: (1) daily interactions, (2) probability of infection, and  (3) duration of illness.  The above assertion shows that making efforts to decrease $R_0$ is essential for controlling propagation of epidemics. Hence measures such as social distancing (self-quarantine, physical separation), washing hands and wearing face coverings, as well as testing / timely-hospitalization can collectively decrease $R_0$.
\end{remark}

\begin{remark}
The above model if further simplified will reduce to the SIR model or SIS model, which is easier to analyze \cite{He00, CB01, HLM14}. It can also be  enriched by dividing into different groups \cite{Ko09, Su10}, or by considering the spatial movement effects \cite{HI95, Ke65}. In this work, we use the SEIR model as a base for our analysis, prediction, and control.
\end{remark}

%H.W. Hethcote.
%The mathematics of infectious diseases.
%SIAM Review, 42(4): 599--653, 2000.
%Brauer, F.; Castillo-Chávez, C. (2001). Mathematical Models in Population Biology and Epidemiology. NY: Springer.

%Hethcote, Herbert W. (1989). Three Basic Epidemiological Models. In Levin, Simon A.; Hallam, Thomas G.; Gross, Louis J. (eds.). Applied Mathematical Ecology. Biomathematics. 18. Berlin: Springer. pp. 119–144. doi:10.1007/978-3-642-61317-3_5.

%Kröger, Martin; Schlickeiser, Reinhard (2020). Analytical solution of the SIR-model for the temporal evolution of epidemics. Part A: Time-independent reproduction factor". Journal of Physics A. doi:10.1088/1751-8121/abc65d. S2CID 225555567.

%[Ko09]A. Korobeinikov, Global properties of SIR and SEIR epidemic models with multiple parallel infectious stages, Bull. Math. Biol. 71 (2009) 75–83.

%[Su10] R. Sun, Global stability of the endemic equilibrium of multigroup SIR models with nonlinear incidence, Computers and Mathematics with Applications 60 (2010) 2286–2291.

%[HI95] Yuzo Hosono and Bilal Ilyas. Traveling waves for a simple diffusive epidemic model. Mathematical Models and Methods in Applied Sciences, 5(07):935--966, 1995.

%[Ke65] David G Kendall. Mathematical models of the spread of infection. Mathematics and computer science in biology and medicine, pages 213--225, 1965.

\subsection{A simple control}
In compact form, the SEIR system may be written as
\begin{align*}
  \dot{U} = F(U; \theta), \quad U(0) &= U_0,
\end{align*}
where
\begin{align*}
  U =[S,E,I,R,D]^\top, \quad   \theta =[\beta,\epsilon,\gamma,\mu]^\top
\end{align*}
are the column vectors of the state and parameters, respectively. Let $g(\cdot)$ be a loss function (to be specified later) at the final time $T$, then the problem of determining the model parameters based on this final cost can be cast as a dynamic control problem:
$$
{\rm min}_\theta \{ g(U(T)) \quad \hbox{subject to }\; \;  \dot U=F(U, \theta),\quad 0<t\leq T; \quad U(0)=U_0; \quad \theta\in\Theta \}.
$$
Here the dependence of $g$ on $\theta$ is through $U(T)$. The set of admissible parameters $\Theta$ may be estimated from other sources. %\red{For Covid-19, $\Theta=\{\theta: 0\le\theta\le1\}$ is appropriate. }
%where $g(\cdot)$ is the loss function to be specified later.

For a short time period, the model parameters are near constant, then we can use gradient-based methods to directly update $\theta$, for which $\nabla_\theta g$ needs to be evaluated.
%optimize $g$
%we require gradients with respect to $\theta$
%The first step is to determine how gradient of the loss depends on the hidden state $U(t)$ at each instant.
We define the Lagrangian functional
\begin{align*}
\mathcal{L}(\theta) & =g(U(T))-\int_0^T ( \dot U -F(U, \theta))^\top V dt,
 %& =L(U(T)) -V(T)U(T)+V(0)U(0)+\int_0^T (V_t U+F(U, \theta)V)dt.
\end{align*}
where $V=V(t)$ is the Lagrangian multiplier depending on time and can be chosen freely, and $U$ depends on $\theta$ through the ODE. A formal calculation gives
\iffalse
\begin{align*}\partial_{\theta_j} g & =\partial_{\theta_j} \mathcal{ L} \\
&= g_U^\top \partial_{\theta_j} U(T)+\int_0^T V(t)^\top \left(F_U \partial_{\theta_j} U(t) +\partial_{\theta_j} F -\partial_{\theta_j} \dot U(t)\right)dt \\
& =   \left( g_U -V(T)\right)^\top \partial_{\theta_j} U(T) + V(0)^\top \partial_{\theta_j} U(0) \\
& \qquad
+\int_0^T ( V(t)^\top F_U +\dot V^\top) \partial_{\theta_j} U(t) +V(t)^\top \partial_{\theta_j} F dt.
\end{align*}
\fi
\begin{align*}
\nabla_{\theta} g & =\nabla_{\theta} \mathcal{L} \\
& = (\nabla_{\theta}U(T))^\top \nabla_U g + \int_0^T \left((\nabla_U F)(\nabla_{\theta} U(t)) +\nabla_{\theta} F -\nabla_{\theta} \dot U(t)\right)^\top V(t) dt \\
& = (\nabla_{\theta}U(T))^\top \left(\nabla_U g -V(T)\right)  + (\nabla_{\theta} U(0))^\top V(0)  \\
& \qquad
+\int_0^T (\nabla_{\theta} U(t))^\top\left((\nabla_U F)^\top V(t) +\dot V(t)\right) + (\nabla_{\theta} F)^\top V(t) dt.
\end{align*}
Thus $\nabla_\theta g$ can be determined in the following steps:
%We now complete the joint method in three steps:
\begin{itemize}
\item Solve the forward problem for state $U$,
\begin{align*}
\dot U(t)=F(U(t), \theta), \quad U(0)=U_0.
\end{align*}
\item Solve the backward problem for co-state $V$,
\begin{align*}
\dot V(t)=- (\nabla_U F)^\top V(t), \quad V(T)=\nabla_U g(U(T)).
\end{align*}
\item Evaluate the gradient of $g$ by
\begin{align*}
\nabla_\theta g=\int_0^T  (\nabla_\theta F)^\top V(t)dt.
\end{align*}
\end{itemize}
In the context of the optimal control theory \cite{PB62}, such $V$ exists and is called the co-state function.
%ollowing the adjoint method  in the optimal control theory.
We note that the $(U, V)$ dynamical system models the optimal strategies for $S, E,  I, R$ populations.
\begin{remark} It is remarkable that computing gradients of a scalar-valued loss with respect to all parameters is done using only state and co-state functions without backpropagating through the operations of the solver.  In addition, modern ODE solvers may be used to compute both $U$ and $V$.
\end{remark}

\subsection{Data-driven optimal control}
%In the above simple control problem, the loss function only depends on the data at the final state, and we assume that the parameters of the SEIR model were constant over the whole time period. While in practice, especially when dynamics lasts a long period of time, a time-varying parameter model and a loss function depends directly on states at multiple observation times would be more realistic.
Now we consider an extended time period with reported data available at intermediate observation times. The data are taken from the reported cumulative infection and death cases.\footnote{Data used here is publicly available in the CSSEGISandData/COVID-19 GitHub repository, which collects data from official sources and organizations.}  In such general case the parameters are typically time-varying, we need to derive a more refined data-driven optimal control. Arranging the data in a vector $U_c = [I_c, D_c]^\top$
at times $0=t_0<t_1<t_2<...<t_n=T$, we aim to
\begin{itemize}
  \item find optimal parameter $\theta(t)$ for $0\le t \le t_{n-1}$ such that the solution to the SEIR system fits the reported data at the grid points $\{t_i\}_{i=1}^{n-1}$ as close as possible, and
  \item find desired parameter $\theta(t)$ for $t_{n-1}\le t\le T$ that is able to control the epidemic spreading at time $T$ at desired values.
\end{itemize}
To achieve these goals, we first define a loss function by
\begin{equation*}
  J=\sum_{i=1}^{n-1}L(U(t_i))+g(U(T)),
\end{equation*}
where
\begin{align}\label{eq:loss}
  L(U(t_i)) &= \lambda_1|I(t_i)-I_c(t_i)|^2+\lambda_2|D(t_i)-D_c(t_i)|^2,\quad 1\le i\le n-1,\\
  g(U(T)) &= \lambda_1|I(T)-I_d(T)|^2+\lambda_2|D(T)-D_d(T)|^2.\notag
\end{align}
Here $\lambda_1, \lambda_2$ are user-defined normalization factors. This loss function is composed of two parts:  the running cost $L$ which measures the error between the candidate solution $(I,D)$ to the SEIR system and the reported data $(I_c, D_c)$ at intermediate times; and the final cost $g$, which measures the error between the candidate solution $(I,D)$ and the desired data $(I_d, D_d)$ at the end time.

We then formulate the task as the following optimal control problem
\begin{equation}\label{eq:opt-control}
\begin{aligned}
  \min \quad &J = \sum_{i=1}^{n-1}L(U(t_i)) + g(U(T)),\\
  \text{s.t.}\quad &\dot{U}=F(U,\theta) \quad t\in(0,T], \quad U(0)=U_0, \quad \theta\in\Theta.
\end{aligned}
\end{equation}
Motivated by the classical optimal control theory, we derive the necessary conditions for the optimal solution to problem \eqref{eq:opt-control}, stated in the following theorem.
%\thmref{alg:pmp}.
\begin{theorem}\label{alg:pmp}
%Consider the optimal control problem \eqref{eq:opt-control}, where $0=t_0<t_1<t_2<...<t_n=T$, $T$ is the final time, $U$ and $\theta$ are the state of the controlled system and control function. $L$ and $g$ are some loss functions measuring the running cost and the cost at the final time, respectively.
Let $\theta^*$ be the optimal solution to problem \eqref{eq:opt-control} and $U^*$ be the corresponding state function, then there exists a piece-wise smooth function $V^*$ and a Hamiltonian $H$ defined by
\begin{equation}\label{eq:hp}
  H(U,V,\theta) = V^\top F(U,\theta)
\end{equation}
such that
\begin{equation}\label{eq:fp}
  \dot{U^*} = F(U^*,\theta^*),\quad t\in(0,T], \quad U^*(0)=U_0,
\end{equation}
\begin{equation}\label{eq:bp}
\begin{aligned}
  %\dot{V^*} &= -(\nabla_UF(U^*,\theta^*))^\top V,\quad t_{i-1}\le t<t_i,\quad i=n,...,1,\\
  \dot{V^*} &= -(\nabla_UF(U^*,\theta^*))^\top V, \quad t_{i-1}\le t<t_i,\quad i=n,...,1,\\
  V^*(T) &=\nabla_Ug(U^*(T)),\\
  V^*(t_i^-) &= V^*(t_i^+) + \nabla_UL(U^*(t_i)), \quad i=n-1,...,1
\end{aligned}
\end{equation}
are satisfied. %$V(t)$is also called the adjoint state of the system.
Moreover, the Hamiltonian minimization condition
\begin{equation*}
  H(U^*, V^*, \theta^*) \le H(U^*, V^*, a) \quad \forall a\in\Theta
\end{equation*}
holds for all time $t\in[0,T]$ but $\{t_i\}_{i=1}^{n-1}$.
%and for all permissible control function $\theta\in\Theta$.
\end{theorem}

\begin{proof}
Recall the classical Bolza optimal control problem \cite{BP07}
\begin{align*}
  \min \quad &J = \int_0^TL(U(t),\theta(t))dt + g(U(T))\\
  \text{s.t.}\quad &\dot{U}=F(U,\theta), \quad t\in[0,T], \quad U(0)=U_0, \quad \theta(t)\in\Theta.
\end{align*}
The Pontryagin's maximum principle states the necessary conditions for optimality: assume $\theta^*$ and $U^*$ are the optimal control function and corresponding state of the system, then there exists a function $V^*$ and a Hamiltonian $H$ defined for all $t\in[0,T]$ by
\begin{equation*}
  H(U,V,\theta) = V^{T}F(U,\theta) + L(U,\theta)
\end{equation*}
such that
\begin{align}
  \dot{U^*} &= \nabla_V H(U^*,V^*,\theta^*),\quad U^*(0)=U_0,\label{eq:state}\\
  \dot{V^*} &= -\nabla_UH(U^*,V^*,\theta^*),\quad V^*(T) = \nabla_Ug(U^*(T)),\label{eq:co-state}
\end{align}
are satisfied. Moreover, the Hamiltonian minimization condition
\begin{equation*}
  H(U^*,V^*,\theta^*) \le H(U^*,V^*,a)
\end{equation*}
holds for all time $t\in[0,T]$ and for all permissible control function $a\in\Theta$.

Notice that the loss function in \eqref{eq:opt-control} can be rewritten as
\begin{equation*}
  J = \int_{0}^{T}\sum_{i=1}^{n-1}L(U(t))\delta(t-t_i)dt + g(U(T)),
\end{equation*}
where $\delta(x)$ is the Dirac delta function, then the corresponding Hamiltonian reads as
\begin{equation}\label{eq:hamiltonian}
  H(U,V,\theta) = V^{T}F(U,\theta) + \sum_{i=1}^{n-1}L(U(t))\delta(t-t_i).
\end{equation}
Hence \eqref{eq:state} can be expressed as \eqref{eq:fp}, \eqref{eq:co-state} has the form of
\begin{equation}\label{eq:co-state-dis}
  \dot{V} = -\sum_{i=1}^{n-1}\nabla_UL(U(t))\delta(t-t_i) - (\nabla_UF(U,\theta))^\top V,\quad V(T)=\nabla_Ug(U(T)).
\end{equation}
Integrate \eqref{eq:co-state-dis} from $t^-_i-\epsilon$ to $t^+_i+\epsilon$ and let $\epsilon$ goes to $0$, then for $i=n-1,...,1$
\begin{align*}
  V(t^+_i)-V(t^-_i)=\nabla_UL(U(t_i)).
\end{align*}
Thus in each interval, \eqref{eq:co-state-dis} reduces to
\begin{equation*}
  \dot{V} = -(\nabla_UF(U,\theta))^\top V,  \quad t_{i-1}\le t<t_i,\quad i=n,...,1,
\end{equation*}
which is exactly \eqref{eq:bp}. Correspondingly, the Hamiltonian \eqref{eq:hamiltonian} reduces to \eqref{eq:hp} for all $t\in[0,T]$ but $\{t_i\}_{i=1}^{n-1}$.
%It is worth mentioning that this condition becomes sufficient if the Hamiltonian $H(x,u,\lambda)$ is convex in $u$.
\end{proof}

\section{Numerical discretization and algorithms}\label{ch:alg}
In this section, we present implementation details to solve problem \eqref{eq:opt-control} via solving the generalized PMP by iteration.
We proceed in the following manner. First we make an initial guess $\theta_0 \in \Theta$. From the control function $\theta_{l}(t)$ in the $l$-th iteration for $l=0, 1, 2,\cdots$,  we obtain  $\theta_{l+1}(t)$ in three steps: \\[5pt]
%a brief description of the method is given below:\\
%\noindent \textbf{Step 1:} Initialize the control function $\theta^{(0)}$, and set $l=0$.\\
\textbf{Step 1:} Solve the forward problem (\ref{eq:fp}) to obtain $U_{l}$.\\
%$\dot{U^{(l)}} = F(U^{(l)},\theta^{(l)}), \quad U^{(l)}(0)=U_0$.\\
%the forward problem \eqref{eq:fp} for $U^{(l)}(t)$.\\
\textbf{Step 2:} Solve the sequence of backward problems (\ref{eq:bp}) to obtain $V_{l}$.\\
%$\dot{V^{(l)}} = -(\nabla_U F(U^{(l)},\theta^{(l)}))^\top V,\quad t_{i-1}\le t<t_i,\quad i=n,...,1,\\
%V^{(l)}(T) =\nabla_Ug(U^{(l)}(T)),\\
%V^{l}(t_i^-) = V^{(l)}(t_i^+) + \nabla_UL(U^{(l)}(t_i)), \quad i=n-1,...,1 $\\
%the backward problems \eqref{eq:bp} for $V^{(i)}(t)$.\\
\textbf{Step 3:} $\theta_{l+1} = \argmin_{\theta\in\Theta} H(U_{l}, V_{l}, \theta, t)$ for each $t\in[0,T]$.\\
%\textbf{Step 4:} Plug in $U^{(i)}(t)$, $V^{(i)}(t)$ and calculate $\nabla_{\theta}H^{(i)}(t)$.\\
%\textbf{Step 5:} Generate a new control function $\theta^{(i+1)}(t)$ with gradient descent method.\\
%\textbf{Step 5:} Set $l\leftarrow l+1$ and go to step 2.\\

This is essentially the method of successive approximations (MSA) \cite{CL82}. %which is an iterative method based on alternating propagation and optimization steps via solving the PMP.
An important feature of MSA is that the Hamiltonian minimization step is decoupled for each $t$. However, in general MSA tends to diverge, especially if the bad initial guess is taken \cite{CL82}.

Our goal is to adopt a careful discretization at each step to
control possible divergent behavior, instead of simply call an existing ODE solver for Steps 1-2, and an optimization algorithm for Step 3.

%Generally, both the forward problem \eqref{eq:fp} and the sequence of backward problems \eqref{eq:bp} can be solved by calling an ODE solver. While in our case, all the state variables $S, E, I, R, D $ have their biological meanings, the solution given by the ODE solver must be positive. Therefore, to ensure the positivity of the solution with no restriction on the choice of step size, we use the numerical discretization method given in \secref{ch:forward} and \secref{ch:backward} to solve the forward and backward problems.

%In step 3, the backward problems \eqref{eq:bp} is solved interval by interval with the same method, thus it suffices to write a scheme for one interval. Also notice that solving the backward problems requires knowing the values of $U$ along the entire trajectory, we thus make the same partition on time for both the forward and backward problems so that the value of $U$ required in step 3 are already calculated in step 2.

%In this work, we take $A=0$ and $d=0$ for the sake of simplicity.
%The time span $[0,T]$ is divided into $n$ intervals.
Within each interval $(t_{i-1}, t_i)$, we approximate both $U$ and $V$ at the same $m$ points with step size
$h=(t_i-t_{i-1})/m$, so that the value of $U$ required in Step 2 are already calculated in Step 1.
% we use $U^{i,k}, V^{i,k}, \theta^{i,k}$ to represent the $k$-th point in the $i$-th interval where $i=1,...,n$ and $k=0,...,m$.
%The end point of the previous interval and the start point of the current interval are the same point.
%Thus, for the forward problem, $U^{i,0}=U^{i-1,m}$,
%we approximate the solution within $(t_{i-1},t_i]$ at $\{U^{i,k}\}^m_{k=1}$; For the backward problem, that is $V^{i,m}=V^{i+1,0}$, we approximate the solution within $[t_{i-1},t_i)$ at $\{V^{i,k}\}_{k=m-1}^0$.
%Below in \ref{ch:forward}, \ref{ch:backward} and \ref{ch:minimization},
We use $U^{i,k}, V^{i,k}, \theta^{i,k}$, where $i=1,...,n$ and $k=0,...,m$ to denote solution values at $k$-th point in the $i$-th interval.

Note that population dynamics (births or deaths) may be neglected at a very crude level on the grounds that epidemic dynamics often occur on a faster time scale than host demography, or we can say heuristically that death of an infected individual and subsequent replacement by a susceptible (in the absence of vertical transmission) is equivalent to a recovery event. Hence, in the numerical study, both $b$ and $d$ are taken to be zero.

%We present only the case when $A=0$ and $d=0$.
Below we discuss the discretization of  the three iteration steps for the case $b=d=0$.

\subsection{Forward discretization}\label{ch:forward}
We focus on the time interval $(t_{i-1},t_i)$. For notation  simplicity, we use $U^k, V^k, \theta^k$ to represent corresponding values at the $k$-th point. We discretize the forward problem \eqref{eq:fp} by an explicit-implicit method
with the Gauss-Seidel type update:
\begin{align*}
  \frac{S^{k+1}-S^k}{h} &= -\beta^kS^{k+1}I^{k}/N^k,\\
  \frac{E^{k+1}-E^k}{h} &= \beta^kS^{k+1}I^k/N^k - \epsilon^kE^{k+1},\\
  \frac{I^{k+1}-I^k}{h} &= \epsilon^kE^{k+1} - (\gamma^k+\mu^k)I^{k+1},\\
  \frac{R^{k+1}-R^k}{h} &= \gamma^kI^{k+1},\\
  \frac{D^{k+1}-D^k}{h} &= \mu^kI^{k+1}.
\end{align*}
This gives
\begin{equation}\label{eqs:forward}
\begin{aligned}
  S^{k+1} &= \frac{S^k}{1+h\beta^kI^{k}/N^k},\\
  E^{k+1} &= \frac{E^k+h\beta^kS^{k+1}I^{k}/N^k}{1+h\epsilon^k}, \\
  I^{k+1} &= \frac{I^k+h\epsilon^kE^{k+1}}{1+h(\gamma^k+\mu^k)}, \\
  R^{k+1} &= R^k + h\gamma^kI^{k+1},\\
  D^{k+1} &= D^k + h\mu^kI^{k+1},
\end{aligned}
\end{equation}
where $k=0,1,...,m-1$, $h=(t_i-t_{i-1})/m$ is the step size. The most important property of this update is its unconditional positivity-preserving property, i.e,
$$
U^k \geq 0 \Rightarrow U^{k+1}>0, \; k=0, \cdots
%that $U^{k+1}$ remains non-negative if $U^k$
$$
irrespective of the size of the step size $h$.
\iffalse
More precisely, we can show that
\begin{align*}
& 0<S^{k+1} \leq S^k,\\
& \frac{E^k}{1+h\epsilon^k}\leq E^{k+1}\leq E^k + S^k-S^{k+1},\\
&   \frac{I^k}{1+h(\gamma^k+\mu^k)} \leq I^{k+1} \\
& R^{k+1} \geq R^k, \\
& D^{k+1} \geq D^k.
\end{align*}
\fi
For the starting value in each interval, we have
$$
U^{1,0}=U(0),  \quad U^{i,0}=U^{i-1,m}, \quad i=2,\cdots, n-1.
$$
\subsection{Backward discretization}\label{ch:backward}
Let $V=[V_S,V_E,V_I,V_R,V_D]^\top$ and notice that
\[
\nabla_UF=
\begin{bmatrix}
-\beta I(N-S)/N^2 &     0     &   -\beta S(N-I)/N^2  & 0 & 0\\
 \beta I(N-S)/N^2 & -\epsilon &    \beta S(N-I)/N^2  & 0 & 0\\
      0    &  \epsilon & -(\mu+\gamma) & 0 & 0\\
      0    &     0     &     \gamma    & 0 & 0\\
      0    &     0     &     \mu       & 0 & 0
\end{bmatrix},
\]
then \eqref{eq:bp} takes the following form
\begin{align*}
  \dot{V_S} &= \beta I(N-S)(V_S-V_E)/N^2,\\
  \dot{V_E} &= \epsilon (V_E-V_I),\\
  \dot{V_I} &= \beta S(N-I)(V_S-V_E)/N^2 + \gamma(V_I-V_R) + \mu(V_I-V_D),\\
  \dot{V_R} &= \dot{V_D} = 0.
\end{align*}
In a similar fashion to the discretization of $U$, we discretize this system by an explicit-implicit method by
\begin{align*}
  \frac{V^{k+1}_S-V^k_S}{h} &= \beta^kI^k(N^k-S^k)(V_S^{k}-V_E^{k+1})/(N^k)^2,\\
  \frac{V^{k+1}_E-V^k_E}{h} &= \epsilon^k(V_E^{k}-V_I^{k+1}),\\
  \frac{V^{k+1}_I-V^k_I}{h} &= \beta^kS^k(N^k-I^k)(V_S^k-V_E^k)/(N^k)^2 + \gamma^k(V^{k}_I-V^{k+1}_R) + \mu^k(V^{k}_I-V^{k+1}_D),\\
  V^{k}_R &= V^{k+1}_R,\qquad V^{k}_D = V^{k+1}_D.
\end{align*}
Then the scheme is given by
\begin{equation}\label{eqs:backward}
\begin{aligned}
  V^{k}_S &= \frac{V^{k+1}_S(N^k)^2+h\beta^kI^k(N^k-S^k)V^{k+1}_E}{(N^k)^2+h\beta^kI^k(N^k-S^k)},\\
  V^{k}_E &= \frac{V^{k+1}_E+h\epsilon^kV^{k+1}_I}{1+h\epsilon^k},\\
  V^{k}_I &= \frac{V^{k+1}_I+h\left(\gamma^kV_R+\mu^kV_D-\beta^kS^k(N^k-I^k)(V_S^k-V_E^k)/(N^k)^2\right)}{1+h(\gamma^k+\mu^k)},\\
  V^{k}_R &= V^{k+1}_R,\qquad V^{k}_D = V^{k+1}_D,
\end{aligned}
\end{equation}
where $k=m-1,...,0$, $h=(t_i-t_{i-1})/m$ is the step size.
For the starting value in each interval, we have
$$
V^{n,m}=\nabla_Ug(U(T)),  \quad V^{i,m}=V^{i+1,0} + \nabla_UL(U(t_i)), \quad i=n-1,\cdots, 1.
$$
Note that the above discretization is well-defined for any $h>0$, hence unconditionally stable since the system is linear in $V$.

%Starting values $V^{i,m}=V^{i+1,0} + \nabla_UL(U(t_i))$. %The most important property of this update is that $U^{k+1}$ remains non-negative if $U^k$ is so, irrespective of the size of the step size $h$.

\subsection{Hamiltonian minimization}\label{ch:minimization}
The Hamiltonian \eqref{eq:hp} is given by
\begin{align*}
  H(U,V,\theta) = -V_S\beta SI/N + V_E(\beta SI/N - \epsilon E) + V_I(\epsilon E - (\gamma+\mu)I) + V_R\gamma I + V_D\mu I.
\end{align*}
After plugging in $U(t), V(t)$ that is obtained by solving the above forward and backward problems with given $\theta(t)$, $H$ may be seen as a functional only with respect to $\theta$. We solve it by the proximal point algorithm (PPA) \cite{Ro76}, that is for any fixed time $t\in(t_{i-1},t_i)$, $i=1,2,...,n,$
\begin{equation}\label{eq:ppa-c}
  \theta_{l+1}(t) = \argmin_{\theta(t)\in\Theta}\left\{H(U_l(t),V_l(t),\theta(t))+\frac{1}{2\tau}\|\theta(t)-\theta_{l}(t)\|^2\right\},
\end{equation}
where $l$ is the index for iteration, $\tau$ is the step size.
\iffalse
Here we apply PPA for two reasons:
\begin{itemize}
  \item It is numerically stable. PPA has the advantage of being monotonically decreasing, which is guaranteed for any step size $\tau>0$.
  \item The objective function in \eqref{eq:ppa-c} is a convex function, which is sufficient for the existence of the solution to the minimization problem.
\end{itemize}
\fi
\begin{remark}
The use of PPA brings two benefits: (1) the objective function in \eqref{eq:ppa-c} is a convex function, which ensures the existence of the solution to the minimization problem; (2) it is numerically stable: PPA has the advantage of being monotonically decreasing, which is guaranteed for any step size $\tau>0$. In this way, the convergence of minimizing the original loss in \eqref{eq:opt-control} with a regularization term is ensured, hence address the convergence issue that MSA generally has \cite{A68}.

\iffalse
\begin{equation*}
  J = \sum_{i=1}^{n-1}L(U(t_i)) + g(U(T)) + \frac{1}{2\tau}\sum_{i=1}^{n}(\theta^{(i)}(t)-\theta^{(i-1)}(t))^2
\end{equation*}
\fi

\end{remark}
Given discretized $U^\alpha$ and $V^\alpha$ with $\alpha=\{i, k\}$,
%$\{U^i\}_{i=1}^{n}$, $\{V^i\}_{i=1}^{n}$,
we solve \eqref{eq:ppa-c} on grid points, that is to solve
\begin{equation}\label{eq:ppa}
  \theta^\alpha_{l+1} = \argmin_{\theta\in\Theta}\left\{H(U^\alpha_l, V^\alpha_l,\theta)+\frac{1}{2\tau}\|\theta-\theta^\alpha_{l}\|^2\right\}.
\end{equation}
Since $H$ is smooth, the above formulation when the constraint is not imposed is equivalent to the following
$$
\theta^\alpha_{l+1}=\theta^\alpha_{l}-\tau \nabla_{\theta}{H}(U^\alpha_l, V^\alpha_l, \theta^\alpha_{l+1}).
$$
\iffalse
We begin with solving the corresponding unconstraint problem, i.e., temporarily ignore the constraint $0\le\theta\le1$. Denote
\begin{equation*}
  \tilde{H}^i_l(\theta) = H(U^i,V^i,\theta)+\frac{1}{2\tau}\|\theta-\theta^i_{l}\|^2.
\end{equation*}
Notice that $\tilde{H}$ is linear with respect to $\theta$, hence by setting $\nabla_{\theta}\tilde{H}^i_l=0$:
\begin{align*}
  \partial_\beta \tilde{H}^i_l &= -S^iI^i(V^i_S-V^i_E)/N^i+(\beta-\beta^i_l)/\tau=0,\\
  \partial_\epsilon \tilde{H}^i_l &= -E^i(V^i_E-V^i_I)+(\epsilon-\epsilon^i_l)/\tau=0,\\
  \partial_\gamma \tilde{H}^i_l &= -I^i(V^i_I-V^i_R)+(\gamma-\gamma^i_l)/\tau=0,\\
  \partial_\mu \tilde{H}^i_l &= -I^i(V^i_I-V^i_D)+(\mu-\mu^i_l)/\tau=0,
\end{align*}
\fi
The special form of $H$ allows us to obtain a closed form solution:
\begin{equation}\label{eqs:theta}
\begin{aligned}
  \beta^\alpha_{l+1} &= \beta^\alpha_l + \tau S^\alpha_l I^\alpha_l((V_S)^\alpha_l-(V_E)^\alpha_l)/N^\alpha_l,\\
  \epsilon^\alpha_{l+1} &= \epsilon^\alpha_l + \tau E^\alpha_l((V_E)^\alpha_l-(V_I)^\alpha_l),\\
  \gamma^\alpha_{l+1} &= \gamma^\alpha_l + \tau I^\alpha_l((V_I)^\alpha_l-(V_R)^\alpha_l),\\
  \mu^\alpha_{l+1} &= \mu^\alpha_l + \tau I^\alpha_l((V_I)^\alpha_l-(V_D)^\alpha_l).
\end{aligned}
\end{equation}
Now taking the constraint $\theta\in\Theta$ into consideration, we simply project the solution \eqref{eqs:theta} back into the feasible region after each iteration, that is
\begin{equation*}
  \theta^\alpha_{l+1} = clip(\theta^\alpha_{l+1}, \underline{\Theta}, \overline{\Theta})
\end{equation*}
where $\underline{\Theta}$, $\overline{\Theta}$ are element-wise lower bound and upper bound of $\Theta$, respectively. That guarantees the  output is constrained to be in $\Theta$.
\begin{remark}
The step size for each parameter may be set differently according to their magnitude scale, and this can indeed improve the training performance, as observed in our experiments.% Also,  %Also note that the gradient w.r.t. $\beta$ in \eqref{eqs:theta} can be seen as $S^\alpha_l (V^\alpha_{S_l}-V^\alpha_{E_l})(I^\alpha_l/N^\alpha_l)$. In the early iteration stage, $I/N$ is very small, which
\end{remark}
%In more general cases where it is not possible to obtain a closed form solution, gradient-based iterative algorithm can be applied to solve \eqref{eq:ppa}, for example the stochastic backward Euler method \cite{YP18}.

\subsection{Algorithms}
The above computing process is wrapped up in Algorithm \ref{alg1}.
%Let $\theta$ be the vector collectively sampled from $\{\theta^\alpha\}$.
Since we can update $\{\theta^{i,k}\}$ simultaneously, we consider $\{\theta^{i,k}\}$ as a vector and still denote it as $\theta$ for notation simplicity. We use $\|\cdot\|$ to stand for the $L_2$ norm.
\begin{algorithm}
\caption{}
\label{alg1}
\begin{algorithmic}[1] % The number tells where the line numbering should start
\Require $\{U_c(t_i)\}_{i=1}^{n}$: data, $t_n=T$: final time, $U_{0}$: initial data at $0$, $\theta_0$: initial guess, $\tau$: step size for the minimization problem
\While{$\|\theta_{l}-\theta_{l-1}\|/\|\theta_{l-1}\|>$ Tol}
\For{$i=1$ to $n$}
\For{$k=0$ to $m-1$}
\State $U^{i,k+1} \leftarrow U^{i,k}$ (solve the forward problem, refer to \eqref{eqs:forward})
\EndFor
\State $U^{i,0} \leftarrow U^{i-1,m}$ (update the initial condition for ODE solver)
\EndFor
\State $V^{n,m} \leftarrow \partial_Ug(U(T))$ (set the initial data for the backward problem)
\For{$i=n$ to $1$}
\For{$k=m-1$ to $0$}
\State $V^{i,k} \leftarrow V^{i,k+1}$ (solve the backward problem, refer to \eqref{eqs:backward})
\EndFor
\State $V^{i,m} \leftarrow V^{i+1,0} + \partial_UL(U(t_i))$ (update the initial condition for ODE solver)
\EndFor
\State $\theta_{l+1} \leftarrow \theta_l$ (solve the minimization problem, refer to \eqref{eqs:theta})
\State $\theta_{l+1} \leftarrow clip(\theta_{l+1}, \underline{\Theta}, \overline{\Theta})$ (ensure $\theta\in\Theta$)
\EndWhile
\State \textbf{return} $\theta, U$
\end{algorithmic}
\end{algorithm}

This data-driven optimal control algorithm provides an optimal fitting to both the data and the SEIR model, it can help to reduce the uncertainty in conventional model predictions with standard data fitting.

%\subsection
{\bf Initialization:}
%In order to formulate an effective iteration algorithm in terms of $\theta$, we need to initialize the parameters.
%To solve \eqref{eq:opt-control}, the first step is
We now discuss how to initialize the control function $\theta$.
A simple choice is to take $\theta_0^{i, k}$ to be a constant vector
for all $1\le i\le n, 1\le k\le m$,
%We initialize them as constants
%during the whole time interval $[0,T]$,
%that is for all $1\le i\le n, 1\le k\le m$, each component of $\theta$ as
%\begin{equation*}
%  \beta^{i,k}_0=C_1,\quad \epsilon^{i,k}_0=C_2,\quad \gamma^{i,k}_0=C_3, \quad \mu^{i,k}_0=C_4,
%\end{equation*}
where the value of %$C_1, C_2, C_3, C_4$
each component relies on a priori epidemiological and clinical information about the relative parameter magnitude. They vary with the area from where the data were sampled.

We can also use the data $\{D_c(t_i)\}_{i=0}^{n}, \{I_c(t_i)\}_{i=0}^{n}$ to obtain a better initial guess for $\mu$.
More precisely, from \eqref{eq:death}, i.e., $\mu=\dot{D}/I$ we take
\iffalse
\begin{equation*}
  \frac{D^{i+1}-D^i}{t_{i+1}-t_i}=\mu^i I^{i+1},
\end{equation*}
such that for $i=1,2,...,n$,  set
\fi
\begin{equation*}
  \mu^{i,k}_0 = \frac{D_c(t_{i+1})-D_c(t_i)}{(t_{i+1}-t_i)I_c(t_{i+1})} \quad \forall k=0,1,...,m.
\end{equation*}
%It is worth mentioning that our method is insensitive to the initial values.
%The coarse way  we initialize the control function $\theta$ does
%help to reduce
%to avoid lots of work that is typically needs to be done to obtain proper initial value
%In step 1 and 2, we also need initial condition to solve the forward %^problem and backward problem.
The initial data for the forward problem is set as
\begin{equation*}
  U_0=[N(0)-I_c(0), 0, I_c(0), 0, D_c(0)]^\top,
\end{equation*}
where $N(0)$ is the initial population of the area analyzed, $I_c(0)$, $D_c(0)$ are the confirmed infections and deaths on the day of the first confirmed cases, respectively.

The initial condition of the backward problem is given in \eqref{eq:bp}. In the present setup, the loss function $g$ only depends on $I$ and $D$, hence the initial condition is given by
\begin{equation*}
  V^{n,m}=[0,0,\partial_Ig(U(T)),0,\partial_Dg(U(T))]^\top.
\end{equation*}

\iffalse
\begin{equation*}
\begin{gathered}
  V_S^{n,m-1}=\partial_Sg(U(T))=0,\quad V_E^{n,m-1}=\partial_Eg(U(T))=0,\quad V_R^{n,m-1}=\partial_Rg(U(T))=0\\
  V_I^{n,m-1}=\partial_Ig(U(T)),\quad V_D^{n,m-1}=\partial_Dg(U(T)).
\end{gathered}
\end{equation*}
\fi
%Finall

%For dynamics that lasts a long period of time, the way we initialize the parameters may not be a good initialization, then much more steps will be needed to converge to the optimal solution. To address this issue, we also propose Algorithm \ref{alg2}, which simply divides the whole time period into several phases and apply Algorithm \ref{alg1} to each phases consecutively.

Finally, we should point out that Algorithm \ref{alg1} with a rough initial guess $\theta_0$ can be rather inefficient when $T$ is large. In such case, we divide the whole interval as $0=t_{n_0}<t_{n_1}<...<t_{n_s}<t_{n_{s+1}}=T$ with $t_{n_j}\in\{t_i\}_{i=1}^{n-1}$ for $1\le j\le s$
%$[0, T]$ into a union of subintervals,
and apply Algorithm \ref{alg1} to each subinterval consecutively. %Summary of implementation steps are given
This treatment is summarized in Algorithm \ref{alg2}.
%Specifically, let the partition be
%$0=T_0<T_1<...<T_s<T$, with $T_j=t_{n_j}\in\{t_i\}_{i=1}^n$,
%$0=t_{n_0}<t_{n_1}<...<t_{n_s}<t_{n_{s+1}}=T$ with $t_{n_j}\in\{t_i\}_{i=1}^{n-1}$ for $1\le j\le s$, %$t_{n_0}=t_0$, $t_{n_{s+1}}=t_n$,
%we first apply Algorithm \ref{alg1} over $[t_{n_0}, t_{n_1}]$ to get $(\theta, U)$, whose value at $t_{n_1}$ can be used as initialization for the next interval $[t_{n_1}, t_{n_2}]$, this moving forward until the last interval $[t_{n_{s}}, t_{n_{s+1}}]$.
%We summarize in Algorithm \ref{alg2}.

%then use the estimated parameters $\theta(T_1)$, predicted data $U(T_1)$ given by the well-trained model and real data $I_c(T_1), D_c(T_1)$ as the initial data so that Algorithm \ref{alg1} can be applied to the next phase $[T_1, T_2]$. Moving forward this way until the dynamics over the whole time period is covered.
%$[S_p(T_1), E_p(T_1), I_c(T_1), R_p(T_1), D_c(T_1)]$ as the initial data where $S_p(T_1), E_p(T_1), R_p(T_1)$ are solutions given by the well-trained model, $I_c(T_1), D_c(T_1)$ are taken from the real data ?
%\iffalse
\begin{algorithm}
\caption{}
\label{alg2}
\begin{algorithmic}[1] % The number tells where the line numbering should start
\Require $\{U_c(t_i)\}_{i=1}^{n}$, $U_{0}$: initial data, $\tau$: step size for the minimization problem.
\Require $\{t_{n_j}\}_{j=0}^{s+1}$, $\theta_0$: initial guess of $\theta$ on $[t_{n_0}, t_{n_1}]$
\For{$j=0$ to $s$}
%\State $\{U_c(t_i)\}_{i=1}^{n} \leftarrow \{U_c(t_i)\}_{i=n_j}^{n_{j+1}}$
\State $\theta, U$ = Algorithm 1($\{U_c(t_i)\}_{i=n_j+1}^{n_{j+1}}$, $[t_{n_{j}}, t_{n_{j+1}}]$,  $U_{0}$, $\theta_0$, $\tau$)
\State $\theta_0, U_{0} \leftarrow \theta(t_{n_{j+1}}), U(t_{n_{j+1}})$
%\State $\theta_0 \leftarrow \theta(T_j)$
\State \textbf{return} $\theta, U$
\EndFor
\end{algorithmic}
\end{algorithm}

\begin{remark} If parameters vary dramatically, $\theta(t_{n_j})$ may not be a good initialization for $\theta$ on $[t_{n_j}, t_{n_{j+1}}]$. When this happens we switch to the rough initial guess and test by trial and error.
%following the same way as we discussed above.
\end{remark}

%It should be pointed out that unlike Algorithm \ref{alg1} which returns a global optimal solution, the solution given by Algorithm \ref{alg2} would be a sequence of local optimal solutions.
%\fi

\begin{figure}[ht]
\begin{subfigure}[b]{0.5\linewidth}
\centering
\includegraphics[width=1\linewidth]{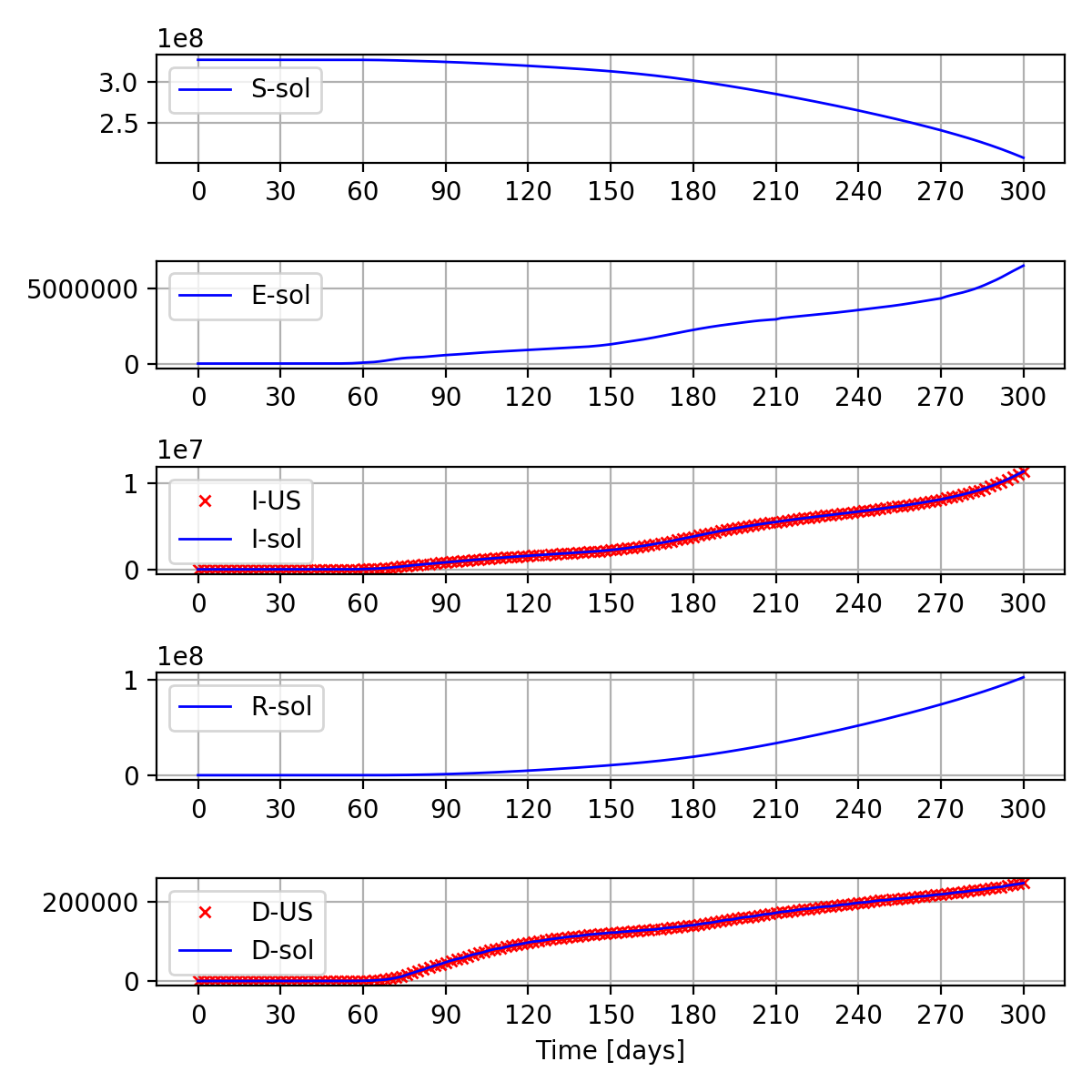}
\caption{}
\end{subfigure}%
\begin{subfigure}[b]{0.5\linewidth}
\centering
\includegraphics[width=1\linewidth]{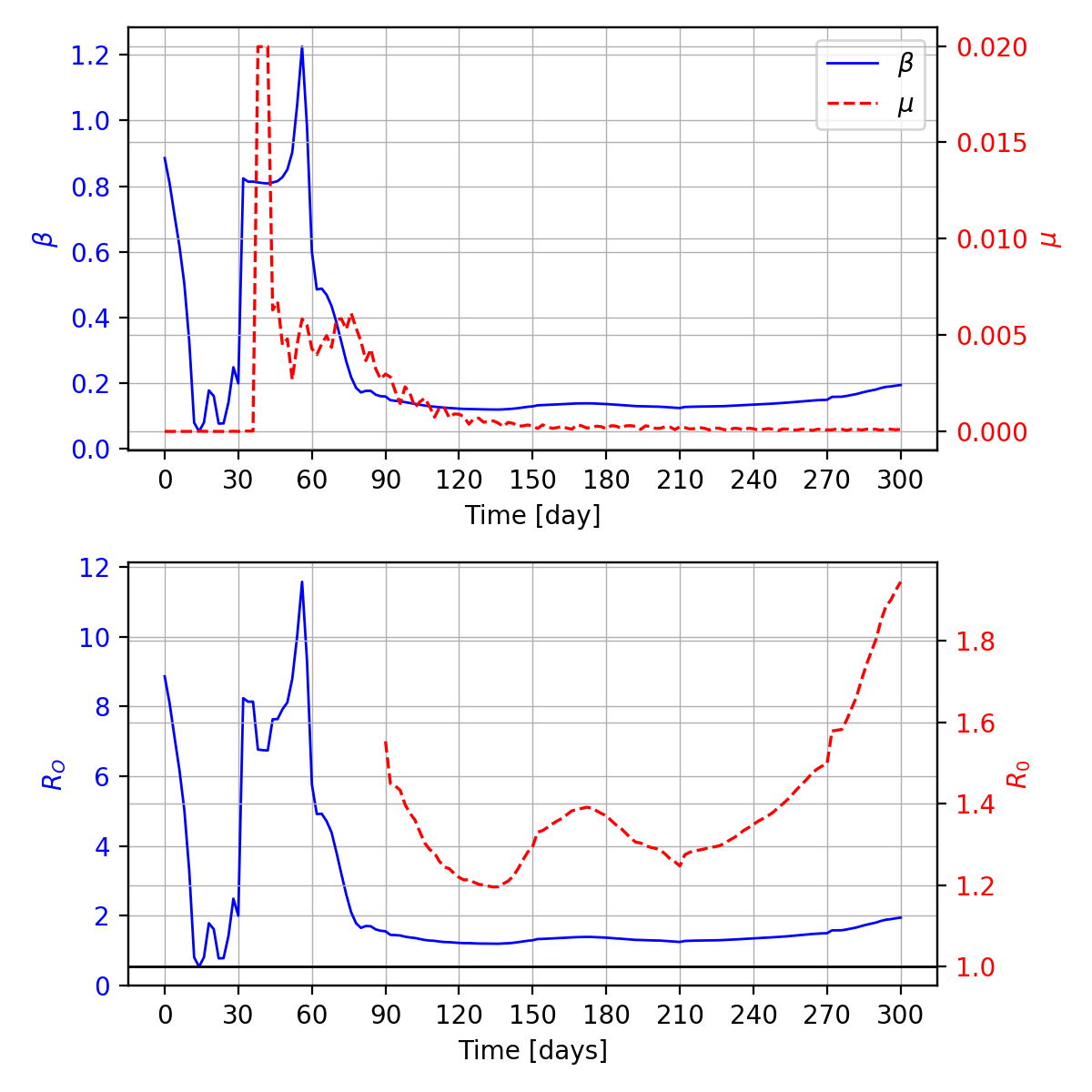}
\caption{}
\end{subfigure}%
\caption{(a) Reported and fitted cumulative infection and death cases in the US (b) Estimated SEIR parameters and the basic reproduction number. $\beta$ ($\mu$) corresponds to the left (right) vertical axis, $\epsilon=0.2$ and $\gamma=0.1$ are almost constant. The dashed line in $R_0$ is a zoomed-in version on the tail of the solid line.}
\label{fig:fitting_us}
\end{figure}%

\section{Experiments}\label{ch:exps}
We now present experimental results to demonstrate the good performance of our algorithms.\footnote{We make our code available at \url{https://github.com/txping/op-seir}} In all the experiments, the normalization factors $\lambda_1$, $\lambda_2$ in \eqref{eq:loss} are chosen %to balance the loss with respect to infection and death cases.
such that the loss with respect to the infection cases is at the same scale as the loss with respect to the death cases.
According to the Centers for Disease Control and Prevention (CDC) in the US, the median incubation period is 4-5 days from exposure to symptoms onset, persons with mild to moderate COVID-19 remain infectious no longer than 10 days after symptom onset, therefore, we set
$$
\Theta=\{\theta \:|\: 0\le\beta\le5,\; 0.2\le\epsilon\le0.25,\; 0.1\le\gamma\le0.2,\; 0\le\mu\le0.01\}.
$$
%$\underline{\Theta}=\{0,\frac{1}{5.8},0.1,0\}$, $\overline{\Theta}=\{3,\frac{1}{4.5},0.2,0.1\}$
For the step size, considering that $\beta$, $\epsilon$ and $\gamma$ are almost at the same scale, which is $100$ times greater than $\mu$, and the permissible range of $\beta$ is much large than $\epsilon$ and $\gamma$, we set $\tau_{\epsilon}=\tau_{\gamma}=\tau$, $\tau_{\beta}=100\tau$ and $\tau_{\mu}=\tau/100$.
%Since our dynamical system modelling COVID-19 transmission is governed by the time-varying parameters, the value of the parameters varying along the time would give some insight

\subsection{Covid-19 epidemic in the US}
As of today [November 17, 2020], it has been about $300$ days since the first infection case of COVID-19 was reported in the US. We thus consider the time period $[0,300]$ and sample the data $[I_c, D_c]$ at $t_i=2i$ for $0< i\le 150$. To apply Algorithm \ref{alg2} we take $\{t_{n_j}\}_{j=0}^{s+1}$ as $\{$0, 30, 60, 90, 150, 210, 270, 300$\}$.
%The reported data $[I_c, D_c]$ are sampled every two days in this whole period.
%The whole period is divided at intermediate time points
%$$
%\{t_{n_j}\}_{j=0}^{s+1} = \{0, 30, 50, 70, 100, 130, 160, 190, 220, 250, 290\}
%$$
%then all the intervals are fitted consecutively.

%To get an qualitative insight into the transmission process of COVID-19, we first estimate the time-varying parameters by fitting the SEIR model to the data.
%We apply Algorithm \ref{alg1} to fit the $290$ days' data in USA. Specifically, the whole time is divided at $0, 30, 50, 70, 100, 130, 160, 190, 220, 250, 290$ and all the intervals are fitted consecutively; the grid spacing $t_{i+1}-t_i$ is set to $2$ to observe more detailed changes (larger grid spacing can be chosen to get smoothing parameters).

\textbf{Data fitting via optimal control with SEIR model.} Figure \ref{fig:fitting_us} (a) shows that our data-driven optimal control algorithm learns the data very well.
In Figure \ref{fig:fitting_us} (b), there is a noticeable peak over the second month, where the value of reproduction number $R_0$ is very large due to a dramatic increase in infection rate $\beta$. After that, $R_0$ goes down to a lower range, with a slight rise around the $6$-th month. Over the last two months we observed another increase in $R_0$. Overall, the value of $R_0$ stays above $1$ and the pattern of $R_0$ is consistent with the increasing trend of the confirmed cases. For a short time period, one may expect the transmission to continue the same way, then the learned model parameters could be used for prediction over a short coming period.

\iffalse
\begin{figure}[ht]
\centering
\includegraphics[width=0.5\linewidth]{plots/pre_us_14.png}
\caption{Prediction for the cumulative number of infection and death cases in the US for the next $14$ days by SEIR model}
\label{fig:pre_us}
\end{figure}%
\fi

%\textbf{Short-time prediction.} For a short time period, one may expect the transmission to continue the same way; hence the obtained model parameters could be used for prediction over a short coming period.
%Here, we use the parameters at the $300$th day ($\beta=0.195$, $\epsilon=0.2$, $\gamma=0.1$, $\mu=1.09\times 10^{-4}$) to make prediction for the cumulative number of infection and death cases in the next $14$ days. For this prediction, we simply exploit the 4th-order Runge-Kutta solver.
%Here we use the 4th-order Runge-Kutta method rather than the discretization method in \secref{ch:forward} to solve the ODE system.
%Figure \ref{fig:pre_us} shows the prediction result for $300-314$ days.

%We shall point out that this prediction is based on our assumption while the assumption itself may not hold sometimes as we can see in \ref{fig:fitting_us} (b). But this prediction can still give us a sense what is going to happen in a near future.
\begin{figure}[ht]
\begin{subfigure}[b]{0.5\linewidth}
\centering
\includegraphics[width=1\linewidth]{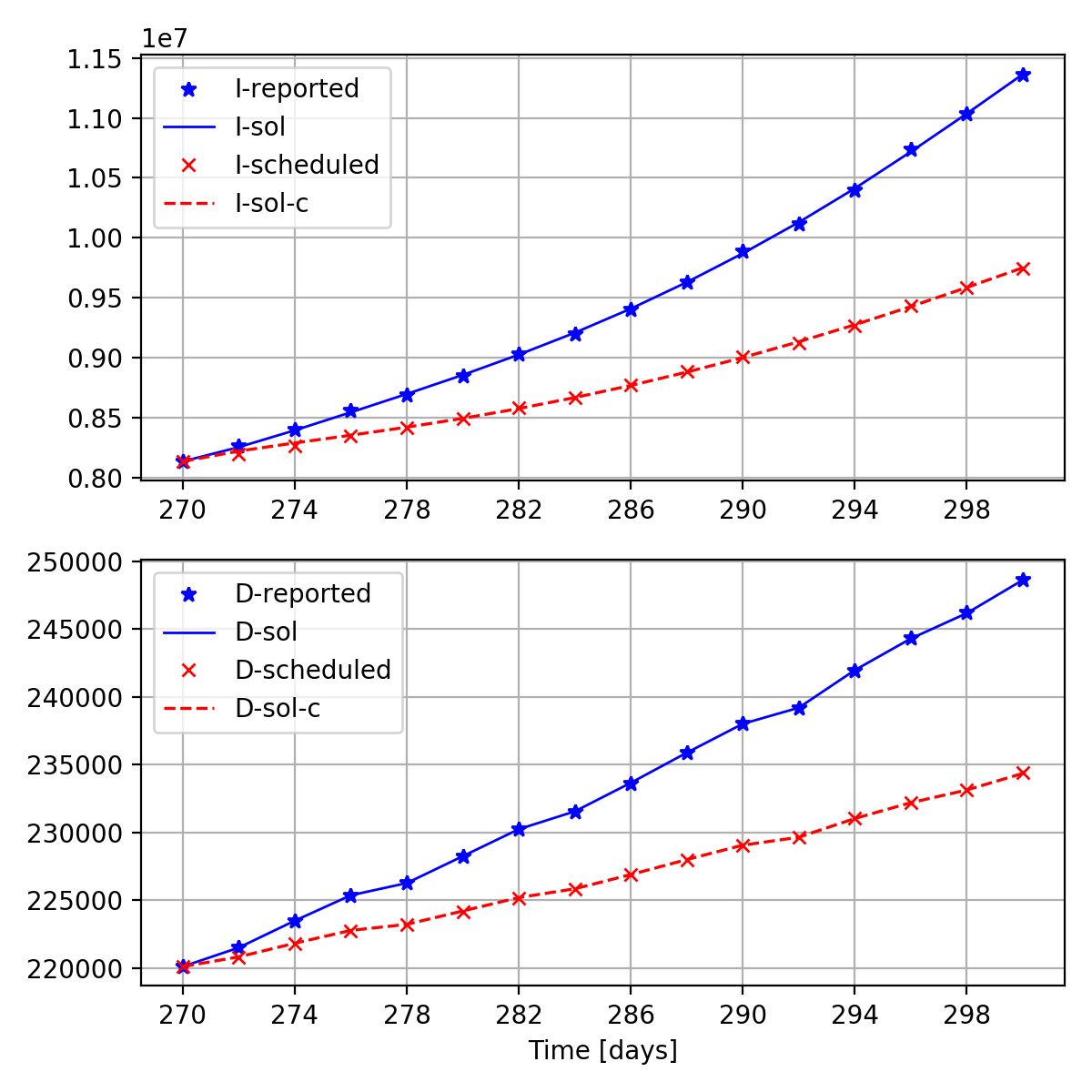}
\caption{}
\end{subfigure}%
\begin{subfigure}[b]{0.5\linewidth}
\centering
\includegraphics[width=1\linewidth]{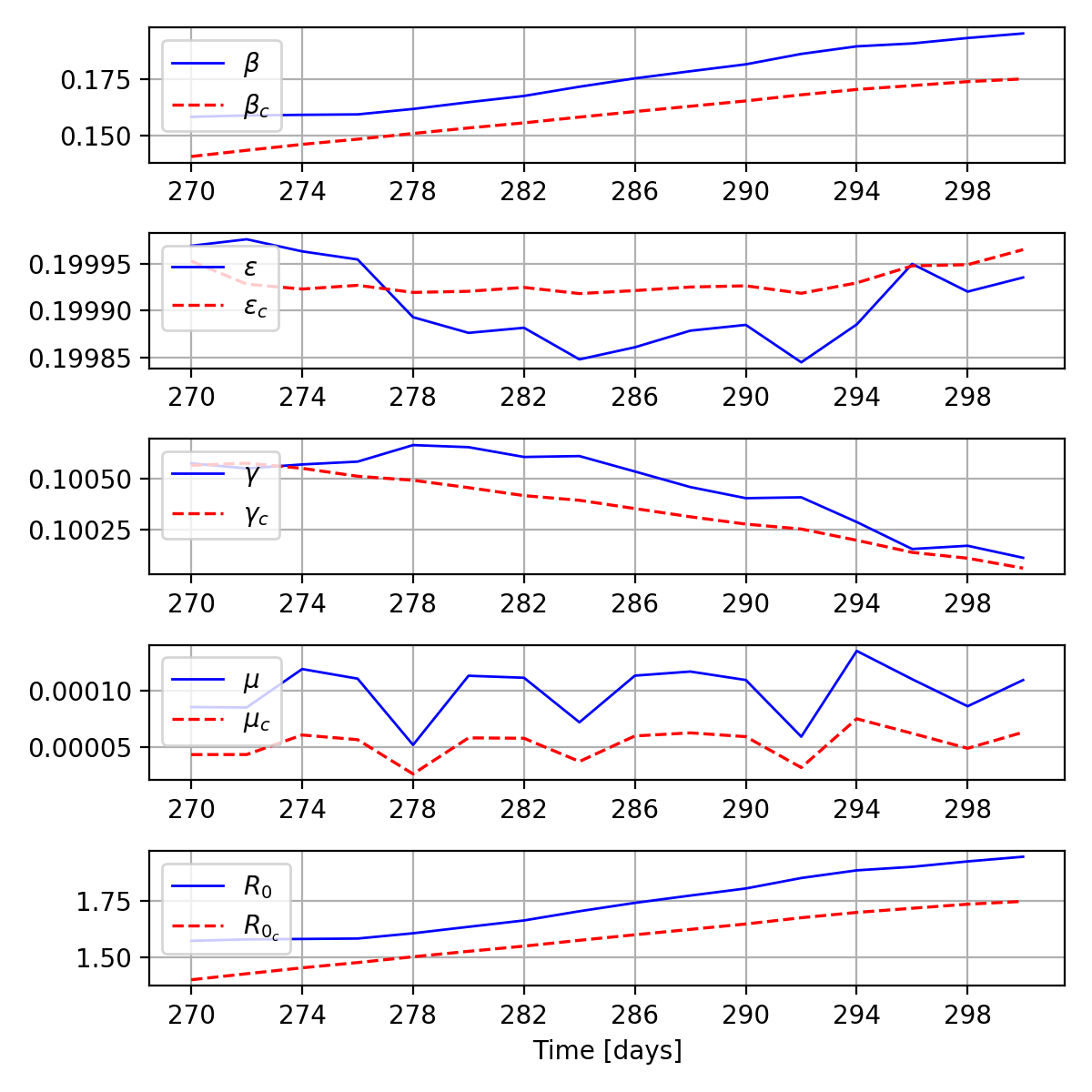}
\caption{}
\end{subfigure}%
\caption{Scheduled control for the US in $270-300$ days by SEIR model}
\label{fig:s_control_us}
\end{figure}%

\textbf{Scheduled control.} From the above prediction result, we see that without interventions, the amount of confirmed and death cases will increase rapidly. We would like to see a slow down of the epidemic spreading as the outcome of various public health interventions. With the present approach, this can be formulated as a scheduled control. Specifically, with a desired level of infection and death cases at the final time $T$, we schedule a sequence of values at intermediate times, from which we apply our algorithm to learn the optimal parameters (control function) such that the state function reaches the desired value at $T$ along the scheduled path.

For instance, starting from the $270$-th day, with the goal of controlling the cumulative number of infection and death cases at $9,746,063$ and $234,390$ on the $300$-th day, respectively, we set a pair of values for each day in the $30$ days as shown in Figure \ref{fig:s_control_us} (a), then learn the parameters from the scheduled data. The results are presented in \ref{fig:s_control_us} (b). Figure \ref{fig:s_control_us} (a) shows that the goal can be achieved by setting the parameters as what have been learned.

To compare the situations with or without a scheduled control, we also present the reported data and corresponding training results in Figure \ref{fig:s_control_us}. In fact, the scheduled intermediate values are obtained by assuming the daily increases were half of the reported daily increases. From Figure \ref{fig:s_control_us} (b), we see that the most significant difference occurs in $\beta$. This can be roughly interpreted as: if the contact rate $\beta$ could be reduced by $0.02$, the number of confirmed cases over the last $30$ days could have been reduced by $50\%$, though the corresponding $R_0$ is still greater than $1$. For virus propagation to eventually stop, $R_0$ needs to be less than $1$, for which $\beta$ must be less than $\gamma+\mu$.

\begin{figure}[ht]
\begin{subfigure}[b]{0.5\linewidth}
\centering
\includegraphics[width=1\linewidth]{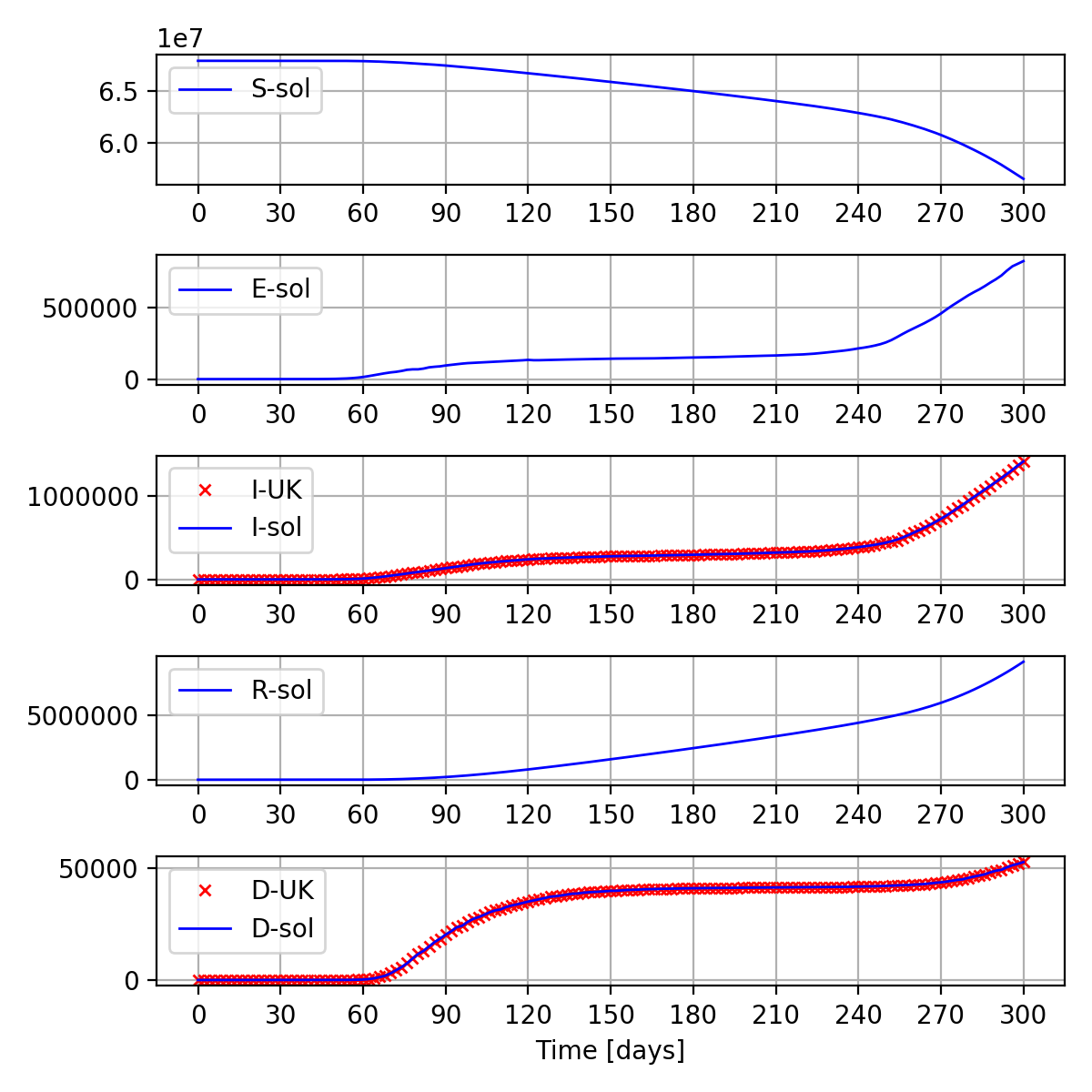}
\caption{}
\end{subfigure}%
\begin{subfigure}[b]{0.5\linewidth}
\centering
\includegraphics[width=1\linewidth]{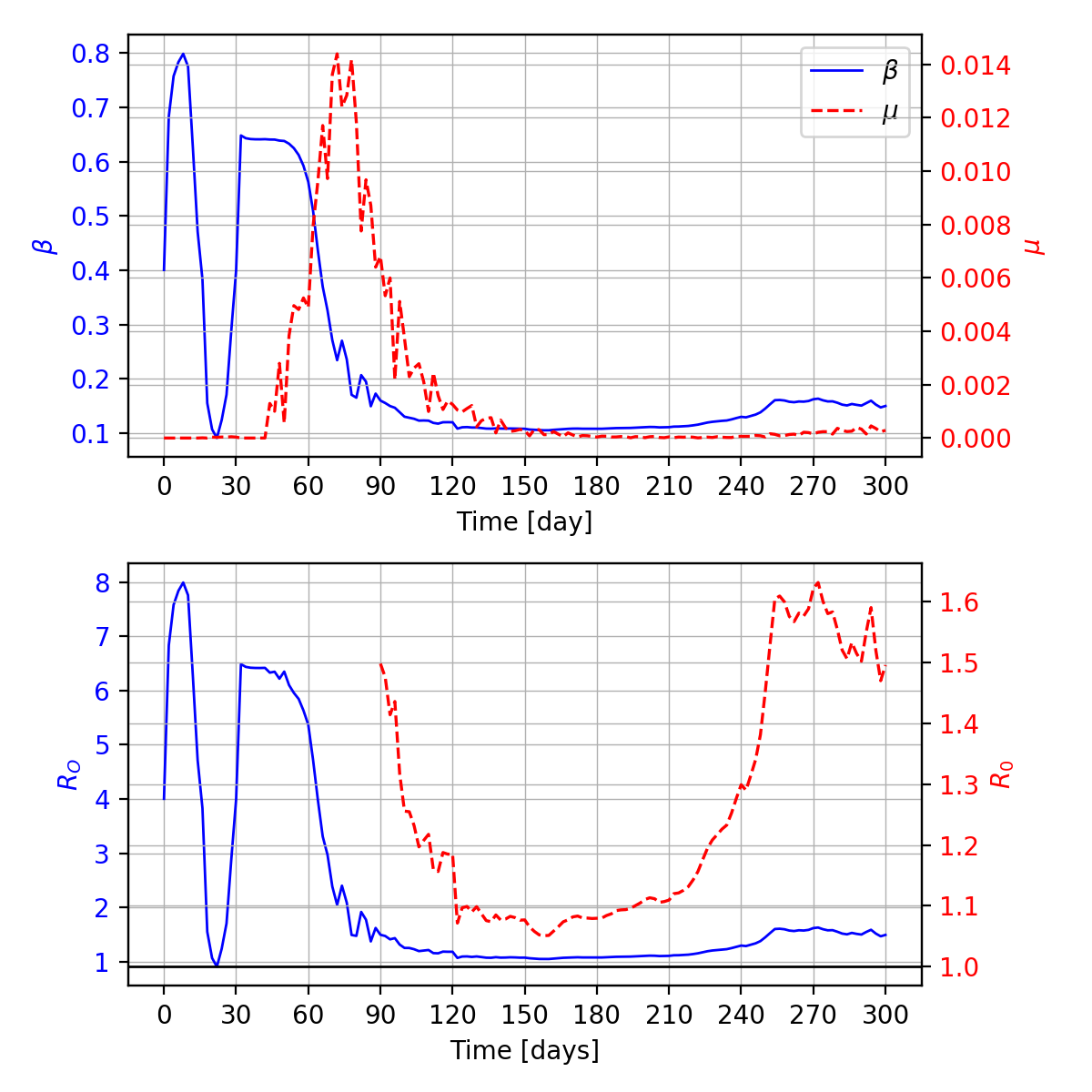}
\caption{}
\end{subfigure}%
\caption{(a) Reported and fitted cumulative infection and death cases in the UK (b) Estimated SEIR parameters and the basic reproduction number. $\beta$ ($\mu$) corresponds to the left (right) vertical axis, $\epsilon=0.2$ and $\gamma=0.1$ are almost constant. The dashed line in $R_0$ is a zoomed-in version on the tail of the solid line.}
\label{fig:fitting_uk}
\end{figure}%

\subsection{Experimental results for other countries} The coronavirus pandemic continues to affect every region of the world, but some countries are experiencing higher rates of infection, while others appear to have mostly controlled the virus. %To test the performance of our algorithms when applied to data in other areas,
In order to see the virus dynamics in other regions, we also provide results for some other selected countries such as the UK, France and China. For the UK, $\{t_{n_j}\}_{j=0}^{s+1}$ are taken as $\{$0, 30, 90, 120, 150, 180, 210, 240, 300$\}$. For France, $\{t_{n_j}\}_{j=0}^{s+1}$ are taken as $\{$0, 30, 60, 90, 180, 300$\}$. For China, $\{t_{n_j}\}_{j=0}^{s+1}$ are taken as $\{$0, 30, 60, 90, 120, 150, 180, 210, 240, 270, 300$\}$.

From Figure \ref{fig:fitting_uk} and \ref{fig:fitting_fc}, we see that the confirmed cases in UK and France display similar patterns. Figure \ref{fig:fitting_cn} shows that China was hit hard early on, but the number of new cases has largely been under control for months.

\begin{figure}[ht]
\begin{subfigure}[b]{0.5\linewidth}
\centering
\includegraphics[width=1\linewidth]{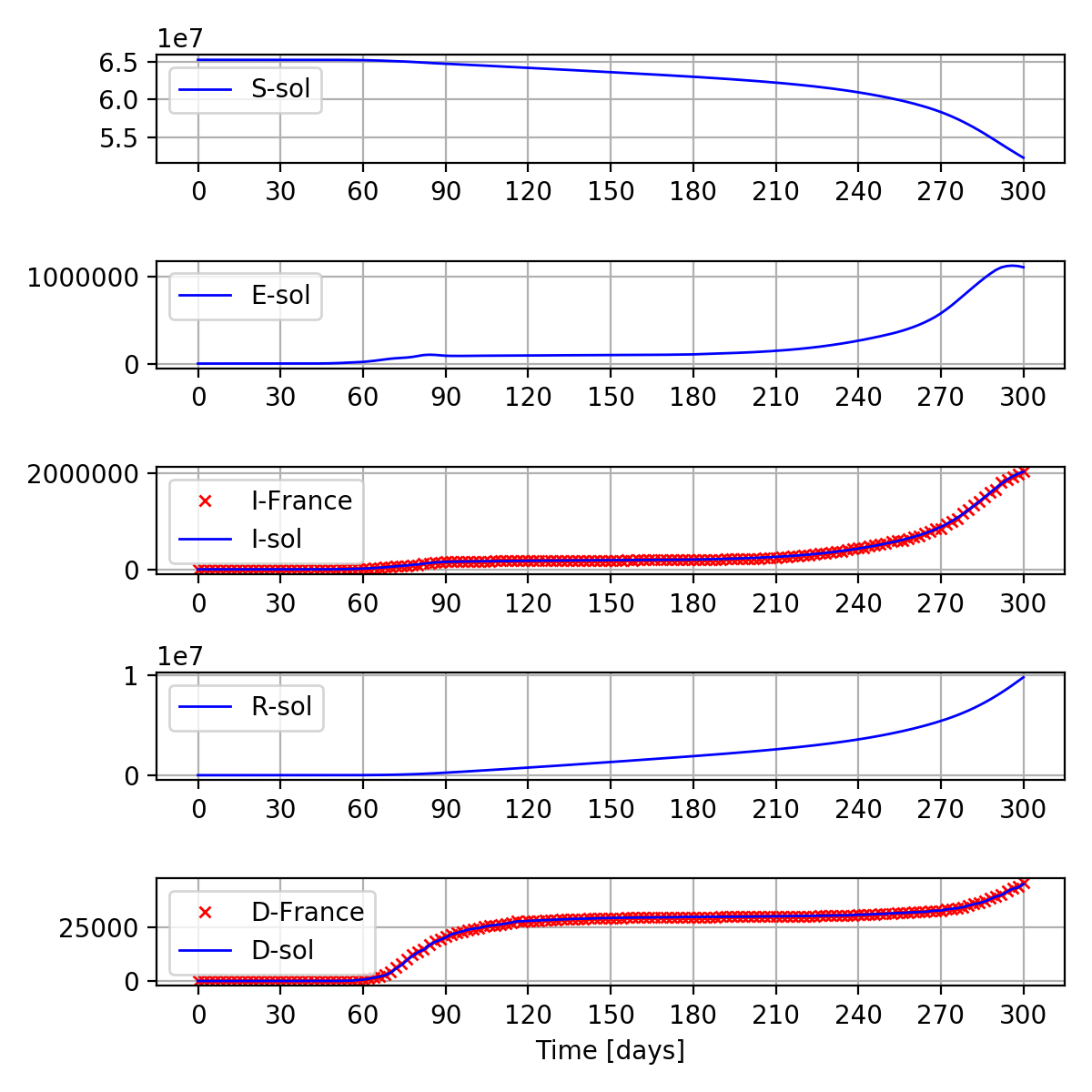}
\caption{}
\end{subfigure}%
\begin{subfigure}[b]{0.5\linewidth}
\centering
\includegraphics[width=1\linewidth]{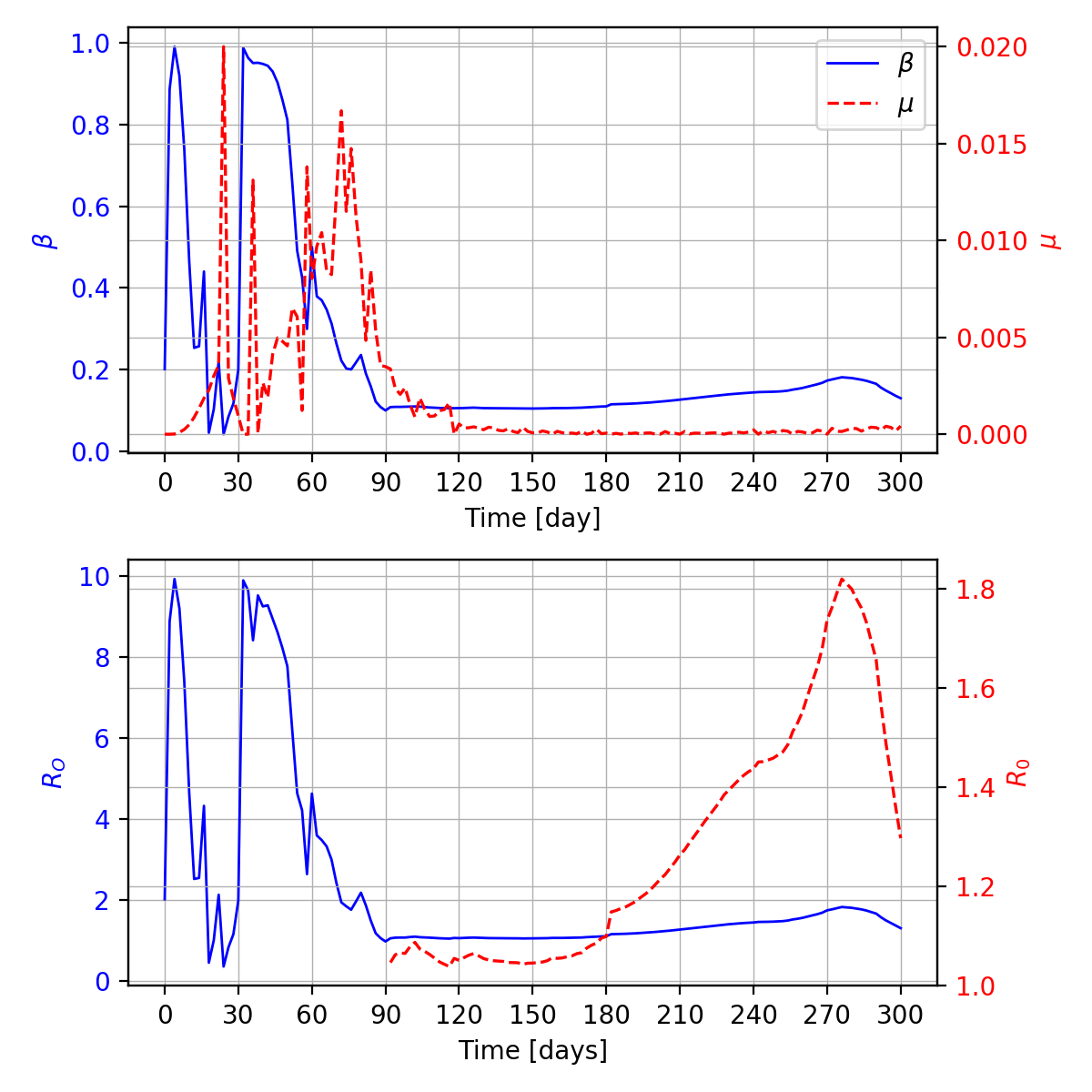}
\caption{}
\end{subfigure}%
\caption{(a) Reported and fitted cumulative infection and death cases in France (b) Estimated SEIR parameters and the basic reproduction number. $\beta$ ($\mu$) corresponds to the left (right) vertical axis, $\epsilon=0.2$ and $\gamma=0.1$ are almost constant. The dashed line in $R_0$ is a zoomed-in version on the tail of the solid line.}
\label{fig:fitting_fc}
\end{figure}%

%appeared to be controlled
%the virus has largely been under control for months.

\begin{figure}[ht]
\begin{subfigure}[b]{0.5\linewidth}
\centering
\includegraphics[width=1\linewidth]{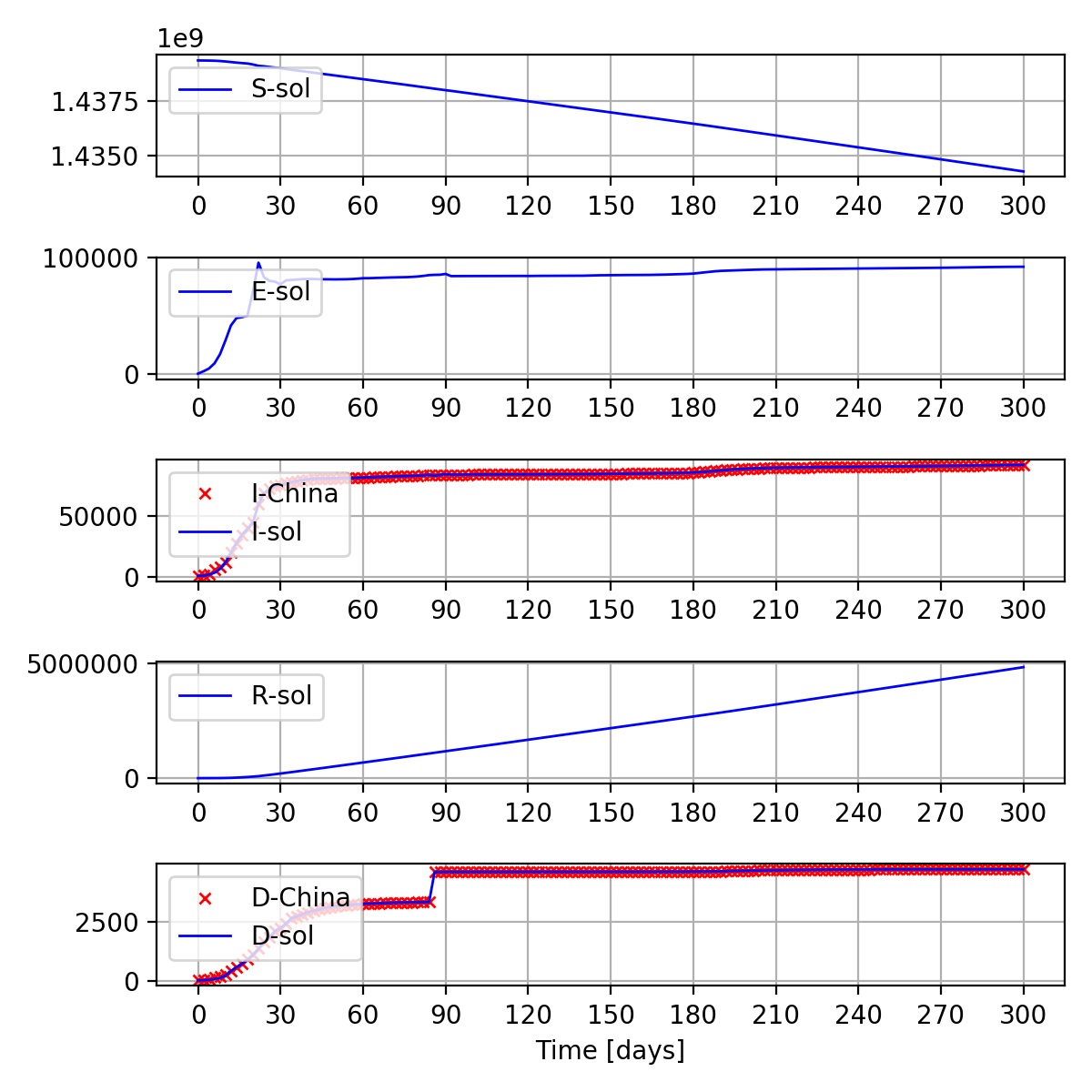}
\caption{}
\end{subfigure}%
\begin{subfigure}[b]{0.5\linewidth}
\centering
\includegraphics[width=1\linewidth]{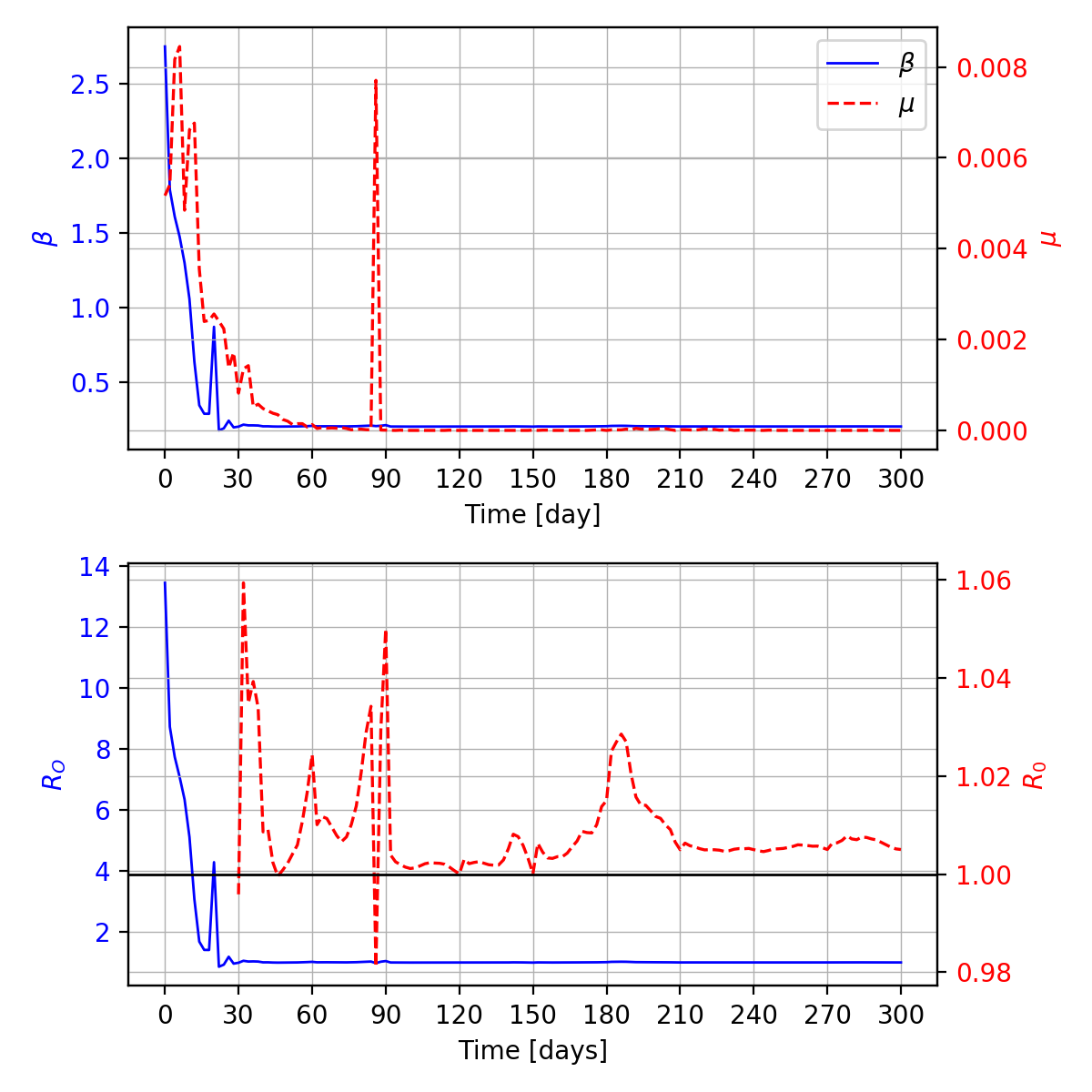}
\caption{}
\end{subfigure}%
\caption{(a) Reported and fitted cumulative infection and death cases in China (b) Estimated SEIR parameters and the basic reproduction number. $\beta$ ($\mu$) corresponds to the left (right) vertical axis, $\epsilon=0.2$ and $\gamma=0.2$ are almost constant. The dashed line in $R_0$ is a zoomed-in version on the tail of the solid line.}
\label{fig:fitting_cn}
\end{figure}%

\section{Discussion}\label{ch:con}
In this paper, we introduced a data-driven optimal control model
%for prediction of the virus spreading among several population groups, which controls the current SEIR model. \red{we introduced a data-driven optimal control approach to
for learning the time-varying parameters of the SEIR model, which reveals the virus spreading process among several population groups. Here the state variables represent the population status, such as $S, E, I,  R, D$  while  the  control variables are rate parameters of transmission between population groups. The running cost is of discrete form fitting the reported data of  infection and death cases at the  observation times.  The terminal cost can be used to quantify the desired level of the total infection  and death cases at scheduled times. Numerical algorithms are derived to solve the proposed model efficiently. Experimental results show that our approach can effectively fit, predict and control the infected and deceased populations.

The data-driven modeling approach presented in this work is applicable to more advanced models such as with spatial movement effects, interaction of different class of populations, for which we need to formulate mean-field game type models over a spatial domain, see \cite{LL20} for the mean-field game formulation of a epidemic control problem.
%CONTROLLING PROPAGATION OF EPIDEMICS VIA MEAN-FIELDGAMES
%WONJUN LEE, SITING LIU, HAMIDOU TEMBINE, WUCHEN LI, AND STANLEY OSHER
%arXiv:2006.01249
On the computational side, our approach involves a non-convex optimization problem, which comes from the multiplicative terms of the SEIR model itself. In future work, we intend to extend our algorithm to more advanced models.
%We also expect to develop and apply AI numerical algorithms to compute models in high spatial dimensions.

\section*{Acknowledgments}
This research was supported by the National Science Foundation under Grant DMS1812666.

%\bigskip
\medskip

\bibliographystyle{amsplain}
\bibliography{ref}

\end{document}